\documentclass[11pt,a4paper,reqno]{amsart}

\usepackage{amsfonts,amsthm,amscd,epsfig,amsmath,amssymb,enumerate}

\numberwithin{equation}{section}

\DeclareMathSymbol{\leqslant}{\mathalpha}{AMSa}{"36} 
\DeclareMathSymbol{\geqslant}{\mathalpha}{AMSa}{"3E} 
\DeclareMathSymbol{\eset}{\mathalpha}{AMSb}{"3F}     
\renewcommand{\leq}{\;\leqslant\;}                   
\renewcommand{\geq}{\;\geqslant\;}                   

\newcommand{\be}{\begin{equation}}

\def\1{\ifmmode {1\hskip -3pt \rm{I}} \else {\hbox {$1\hskip -3pt \rm{I}$}}\fi}

\newtheorem{Th}{Theorem}[section]
\newtheorem{Le}[Th]{Lemma}

\newtheorem{Cor}[Th]{Corollary}

\newcommand{\cA}{\ensuremath{\mathcal A}}
\newcommand{\cB}{\ensuremath{\mathcal B}}
\newcommand{\cC}{\ensuremath{\mathcal C}}
\newcommand{\cD}{\ensuremath{\mathcal D}}
\newcommand{\cE}{\ensuremath{\mathcal E}}

\newcommand{\cJ}{\ensuremath{\mathcal J}}
\newcommand{\cK}{\ensuremath{\mathcal K}}
\newcommand{\cL}{\ensuremath{\mathcal L}}
\newcommand{\cM}{\ensuremath{\mathcal M}}

\newcommand{\cQ}{\ensuremath{\mathcal Q}}
\newcommand{\cR}{\ensuremath{\mathcal R}}
\newcommand{\cS}{\ensuremath{\mathcal S}}

\newcommand{\cU}{\ensuremath{\mathcal U}}

\newcommand{\bbR}{{\ensuremath{\mathbb R}} }

\let\a=\alpha \let\b=\beta  \let\c=\chi \let\d=\delta  \let\e=\varepsilon
 \let\g=\gamma     \let\k=\kappa  \let\l=\lambda
          \let\p=\pi  
\let\r=\rho  \let\s=\sigma \let\t=\tau   
  
\let\D=\Delta   \let\G=\Gamma

\title[Averaging and large deviation principles for PDMPs]{Averaging and large deviation principles for fully--coupled
piecewise deterministic Markov processes and  applications to
molecular motors}

\author{A. Faggionato}
\address{Alessandra Faggionato. Dipartimento di Matematica ``G. Castelnuovo", Universit\`a ``La
  Sapienza''. P.le Aldo Moro  2, 00185  Roma, Italy. e--mail:
  faggiona@mat.uniroma1.it}
\author{D. Gabrielli}
\address{Davide Gabrielli.  Dipartimento di Matematica, Universit\`a
dell'Aquila,   67100 Coppito, L'Aquila, Italy. e--mail:
gabriell@univaq.it}
\author{M. Ribezzi Crivellari}
\address{Marco Ribezzi Crivellari. Dipartimento di Fisica,  Universit\`a  di Roma
Tre, Via della vasca navale 84, 00146 Roma . e--mail:
ribezzi@fis.uniroma3.it}

\begin{document}

\maketitle

\begin{abstract} We consider  Piecewise Deterministic Markov Processes (PDMPs) with a finite set of discrete
states. In the regime of fast jumps between discrete states, we
prove a law of large number and a large deviation principle. In the
regime of fast and slow jumps, we analyze a coarse--grained process
associated to the original one and prove its convergence to a new
PDMP with effective force fields and jump rates. In all the above
cases, the continuous variables evolve slowly according to ODEs.
 Finally, we discuss some applications related
to the mechanochemical cycle of macromolecules, including
strained--dependent power--stroke molecular motors. Our analysis
covers the case of fully--coupled slow and fast motions.

\bigskip


\smallskip

\noindent
{\em MSC-class:   34C29, 60F10, 80A30}

\end{abstract}



\section{Introduction}

Several systems in physics, chemistry,  biology, control and
optimization theory  present a multiscale character due to
interacting parts evolving on different relevant timescales. In the
case of two relevant timescales, the dynamics is a combination of
slow and fast motions and often  the slow motion is supposed to be
well approximated by averaging the effect of fast motion,
considering the fast variables as locally equilibrated. This is the
content of the {\sl averaging principle}, which  has been
successfully applied in several contests. We just mention  few
examples.    One of its first applications has been the study, due
to Newton, of the precession of the equinoxes (see \cite{A} for
averaging methods in mechanics). An example in quantum mechanics is
given by the Born-Oppenheimer approximation, based on the fact that
 electrons move much faster than nuclei. Recently, averaging methods
have been applied also to model climate--weather interactions (see
for example \cite{CMP}).

\smallskip

 The validity of the   averaging principle has been rigorously proved  for several
models assuming that the fast motion is independent of the slow one,
assumption which is not fulfilled by many real systems where fast
and slow variables influence each other. A rigorous proof in a
stochastic  fully--coupled case, with additional results on large
deviations, has been provided by \cite{V1} and \cite{V2} (the former
contains errors, recently corrected in the latter). There the author
considers a system described by variables $(x_t,\s_t)\in M\times N$,
$M$ and $N$ being differential manifolds, such that the slow
variable $x_t$ evolves according to an ODE  depending both on $x_t$
and $\s_t$, while the fast variable $\s_t$ undergoes a diffusion
with local drift and variance depending again both on $x_t$ and
$\s_t$. In \cite{K} the author has given among other an alternative
proof of the averaging and large deviation principles for the
fully--coupled model as above, with the extension   that the  fast
variable $\s_t$ varies in $N\times \G$, $N$ being a differentiable
manifold and $\G$ a finite set, by diffusion on $N$ and random jumps
in $\G$.

\smallskip

 We consider here a fully--coupled system whose state is described by
a slow variable $x_t$ varying in a differentiable manifold $M$ (for
simplicity of notation we take below  $M=\bbR^d$) and by a  fast
variable $\s_t$ varying in a finite set $\G$. We will consider also the case that $\s_t$ evolves
fast only when varying inside some subclasses -- metastates -- of $\G$.
We call $x_t$ the
mechanical state of the system, and $\s_t$ the chemical one. The
mechanical state $x_t$  evolves according to an ODE whose form
depends on the chemical state $\s_t$, while $\s_t$ is a
continuous--time stochastic  process with jump rates depending on the
mechanical state $x_t$. In other words, the system we consider  is a
fully--coupled piecewise deterministic Markov process (PDMP).
  Our
motivation and terminology  come from molecular motors, but the same
model is natural in other contests as for example operations
research, control theory  and optimization (see \cite{D2},
\cite{YZ}).

 \smallskip
 Molecular motors are proteins, hence biological
macromolecules, working inside the cell at the nanometer scale as
microscopic engines, powered by the  chemical energy provided by ATP
hydrolysis. Their tasks are various, as for example they can
transport cargos along filaments or determine  muscle contraction.
 Molecular motors can be in different structural states, that we
simply call {\em chemical states}: attached to the  filament,
detached from the filament, bound   to an ATP molecule, bound to the
products of ATP hydrolysis and so on. Independently from the
chemical state, there is  a specific region of the molecular motor
working as a lever--arm, able to swing  and produce work,  while the
fuel (ATP) is collected from the environment and used to derive
energy.  The conformation of the lever--arm can be described by a
continuous variable $x \in M$, that we call {\em mechanical state}.
The above picture is known in the biophysical literature as
power--stroke model, cross--bridge model and also lever--arm model,
in contrast with the so called ratchet model  (see \cite{H} and
references therein).

\smallskip

  In order to investigate the functionality of molecular motors, it is enough to describe their state by means of the
pair  $(x,\s)$, whose dynamics  can be modeled  in a first
approximation by a fully--coupled  PDMP  as discussed in Section
\ref{motori}. There is experimental evidence that the family $\G$  of
chemical states is partitioned in subclasses $\G_1, \G_2, \dots ,
\G_\ell$ inside which chemical jumps are much faster (see \cite{D} and references therein).
As implicitly assumed in biophysical papers as \cite{VD}, we  also suppose that these  jumps are faster than
the mechanical relaxation time scale.
 Considering
first the case  $\ell=1$, in the limit of high frequency of chemical
jumps,  we prove averaging and large deviation principles  both for
the slow motion of the mechanical state $x$ and for the occupation
measure associated to the chemical states visited in a fixed time
interval.
 The analysis for  the occupation measure is lacking in \cite{K},
moreover we provide a different  and  simpler  derivation of the
averaging and large deviation principles with respect   to
\cite{V1}, \cite{V2} and \cite{K}.
 In the case $ \ell >1$,   the subclasses $\G_j$ can be
treated as chemical metastates. We then show   that  in the limit of
high jump--frequency   the coarse--grained    process describing the
evolution of the mechanical state and the chemical metastate of the system weakly
converges  to a new PDMP with effective force fields (ODE's) and
transitions rates. In Section \ref{motori} we discuss some
biological consequences concerning molecular motors, while in a
companion paper \cite{FGR} we further investigate the model from a
more physical viewpoint.

\smallskip

\section{Model and main results}
\label{mod}

We  consider systems whose states are described by a pair $(x,\s)\in
\bbR^d\times \G$,  where $\G=\{\s_1, \s_2, \dots, \s_{|\G|}\}$ is a
finite set. The variable $x$ will define the {\sl mechanical state}
of the system and the variable $\s$ its {\sl chemical state}.


\smallskip

When the chemical state is $\s$, the system evolves mechanically
according to the ODE
\begin{equation}\label{sole}
\dot{x}(t)= F_\s (x(t),t)
\end{equation}
where the force field $F_\s:\bbR^d \times [0,\infty) \rightarrow
\bbR ^d $ is a continuous function,  locally  Lipschitz w.r.t. $x$.   This means that, given $T>0$ and a compact subset
$\cK\subset \bbR^d$, there exists a positive
constant $K$ such that
\begin{equation}\label{tragediegreche}
\left| F_\s (x,t) - F_\s (x',t) \right| \leq K |x-x'|\,, \qquad
\forall x,x' \in \cK,\, t \in [0,T]\,.
\end{equation}
Below we will specify some additional assumption on $F_\s$ assuring that  the solution $x(t)$
of (\ref{sole}) starting in a given  point $x\in \bbR^d$ is uniquely
determined in each time interval $[0,T]$.
\smallskip

In order to specify how the system jumps from one chemical state to
the other, for each $\s \in \G$  let $\g_\s (x,t)$ be a measurable
positive function on $\bbR^d \times [0,\infty)$ and let $p (\s,\s'
|x,t)$ be a function defined on $\G\times \G \times \bbR^d \times
[0,\infty) $
with the following properties:\\
(i) $\forall \s ,\s' $, $p (\s,\s' |\cdot,\cdot)$ is
measurable,\\
(ii) $\forall x, t,\s $, $p(\s, \cdot |x,t)$ is a probability
measure on $\G$ such that  $p(\s,\s |x,t)=0$.


\smallskip
The mechanochemical evolution of the system is described as follows.
Suppose the system starts at time zero in the state $(x_0, \s_0)$.
Call $x_0(t)$ the solution of (\ref{sole}) with $\s=\s_0$, such that
$x_0(0)=x_0$. Let $\t_1$ be a  random variable with value in
$(0,\infty]$ s.t.
$$
P(\t_1>t)= \exp \left(-  \int_0 ^t \g_{\s_0 }(x_0(s),s) ds
\right)\,.
$$
Note that the r.h.s.  decreases in $t$ and  equals $1$ if $t=0$,
thus implying that $\t_1$ is well defined. Then, in the random time
interval $[0,\t_1)$ the state of the system is given by $(x_0(t),
\s_0)$. If $\t_1=\infty$, we have done. Otherwise, choose a chemical
state $\s_1\in \G$ with probability $p ( \s_0,\s_1
|x_0(\t_1),\t_1)$, and call $x_1(t)$  the solution on $[\t_1,
\infty)$ of the Cauchy problem
$$
\begin{cases}
\dot{x}_1 (t) = F_{\s_1} ( x_1(t),t)\,, \qquad t \geq \t_1\,, \\
x_1(\t_1)=x_0(\t_1)\,.
\end{cases}
$$
Now let $\t_2$ be a  random variable with values in $(\t_1,\infty]$,
such that
$$
P(\t_2>t)=
 \exp \left(- \int_{\t_1} ^t \g_{\s_1
}(x_1(s),s) ds \right)\,,  \qquad t \geq \t_1\,.
$$
Then in the time interval $[\t_1,\t_2)$ the state of the system is
given by $(x_1(t), \s_1)$.

In general, denoting by  $\t_k$ the time of the $k$--th chemical
jump and by $(x_k(t), \s_k)$ the evolution  of the system in the
time interval $[\t_k, \t_{k+1})$, one has that $\t_{k+1}$ is a
random variable with value in $(\t_k,\infty]$ such that
\begin{equation}\label{bieta}
P(\t_{k+1}>t)=  \exp \left( - \int_{\t_k} ^t \g_{\s_k }(x_k(s),s) ds
\right)\,, \qquad t \geq \t_k\,.
\end{equation}
Moreover, $\dot{x} _k(t) = F_{\s_k} ( x_k (t) , t)$ on $[\t_k,
\t_{k+1})$ and  at time $\t_{k+1}$ the system jumps to a new
chemical state $\s_{k+1}\in \G$ with probability $p (\s_k,
\s_{k+1}|x_k(\t_{k+1}), \t_{k+1})$, while the mechanical state
remains the same, i.e. $x_{k} (\t_{k+1})=x_{k+1}(\t_{k+1})$. Setting
$\t_0=0$, the evolution $\bigl( x(t), \s(t)\bigr)$ is then defined
as
\begin{equation}
\bigl( x(t), \s(t)\bigr)= (x_k(t), \s_k) \,, \qquad \text{ if } \;
\: \t_k \leq t < \t_{k+1}\,.
\end{equation}
Note that the path $x(t)$ is continuous and piecewise $C^1$. In
order to have a well defined dynamics over $[0,T]$, it is necessary
that a.s. the system makes a finite number of jumps in the time
interval $[0,T]$. We will discuss below under which conditions this
automatically happens.

\smallskip

 As proven in \cite{D1}, \cite{D2}, the above stochastic
process is a strong Markov process, called {\sl piecewise
deterministic Markov process} (PDMP) with $(x,t)$--dependent
characteristics $\left( F,p,\g\right)$, defined as
$$
F=\left\{ F_\s\right\}_{\s\in \G};\qquad p=
\left\{p(\s,\s'|\cdot,\cdot)\right\}_{\s,\s'\in \G};\qquad
\g=\left\{\g _\s\right\}_{\s\in \G}.
$$
The time dependent generator $L_t$ reads
\begin{equation}\label{generare}
L_t g (x,\s) = F_\s (x,t) \cdot \nabla_x g(x, \s) + \g_\s (x,t) \sum
_{\s' \in \G} p (\s,\s'|x,t) \bigl ( g (x,\s') - g(x, \s) \bigr)
\end{equation}
for   functions  $g : \bbR^d \times \G \rightarrow \bbR$ regular in
$x\in \bbR^d$. A characterization of the domain of the so called
{\sl extended generator} of the process is given in \cite{D1},
\cite{D2}. Defining the transition rate
$$
r(\s,\s'|x,t):= \g_\s(x,t)  p(\s,\s'|x,t)\,.
$$
for a chemical jump from $\s$ to $\s'$ at time $t$ when being in the
state $(x,\s)$, the above generator can be written as
\begin{equation}\label{generale}
L_t g (x,\s) = F_\s (x,t) \cdot \nabla_x g(x, \s) +  \sum _{\s' \in
\G} r (\s,\s'|x,t) \bigl ( g (x,\s') - g(x, \s) \bigr)\,.
\end{equation}
Since knowing  $\bigl( F, r \bigr)$ is equivalent to knowing $\left(
F, p, \g  \right)$, we call both these families of functions {\sl
characteristics} of the PDMP.

\smallskip

 Given $x,t$ we define $L_c (x,t)$ as the contribution to
$L_t$ coming only from the  chemical transitions at the mechanical
state $x$, i.e. $L_c(x,t)$ is the operator on $\G^\bbR$ s.t.
$$
L_c (x,t) g(\s)= \sum _{\s' \in \G} r (\s, \s'|x,t) \bigl ( g (\s')
- g( \s) \bigr)\,, \qquad g :\G\rightarrow \mathbb R.
$$
For fixed $x,t$,   $L_c (x,t) $ is the  generator of the
time--homogeneous   Markov chain  on the space $\G$ which jumps from
$\s$ to $\s'$ with probability $  p (\s, \s' |x,t)$ after having
waited an exponential time with parameter $\g _\s (x,t)$ in the
state $\s$ (note that here $t$ has to be thought of as  a fixed
parameter).

\smallskip

We collect here our technical assumptions recalling that $[0,T]$ is
the time interval on which the evolution of the system will be
observed.

\begin{itemize}

\item  Assumption (A1):   a.s. the number  of chemical jumps
in the interval $[0,T]$ is  finite.

\item Assumption (A2): for any pair $(x,t) \in \bbR^d \times
[0,T]$, the time--homogeneous  Markov chain on $\G$ with generator
$L_c (x,t)$ is ergodic, i.e. it  visits with positive probability
any state in $\G$, for any starting point. We call $\mu(\cdot |x,t)$ its unique
invariant probability measure on $\G$.

\item Assumption (A3): the transition rates $r(\s,\s'|\cdot,\cdot)$
are nonnegative  functions and  belong to  $C^1( \bbR^d \times
[0,T])$.

\item Assumption (A4):   each force field
$F_\s(x,t) $
is a  continuous function in
$(x,t) \in \bbR^d \times [0,T]$ and is locally  Lipschitz  in  $x$, i.e.  for each compact subset $\cK\subset
\bbR^d$ there exists a constant $K>0$ such that
\begin{equation}
 |F_\s(x,t)-F_\s(y,t)|\leq K|x-y|\,, \qquad
\forall x,y \in \cK,\; t \in [0,T]\,.
 \label{lipF}
\end{equation}
 Moreover, we assume that there exist constants $\k_1,\k_2>0$
such that
\begin{equation}\label{muccamuu}
|F_\s (x,t) |\leq \k_1+\k_2 |x|\,, \qquad \forall x \in \bbR^d,\, t \in [0,T] \,.
\end{equation}

\end{itemize}
Only in Section \ref{luna}, when considering fast and slow chemical
jumps,  we will slightly change  assumption (A2).

\smallskip

Let us give some comments on the above assumptions.
Assumption (A1) implies that the chemical evolution  is well
defined a.s.
A simple criterion assuring (A1) is the following:
\begin{equation}
\sup_{\s \in \G, x \in \bbR^d, t \in [0,T]} \g_\s(x,t) < \infty\,.
\end{equation}
Indeed, calling $c$ the l.h.s., the random  number of  chemical
jumps in the interval $[0,T]$ is dominated by $N_T$, $N_\cdot$ being
a Poisson point process with rate $c$. This follows at once
observing that, due to (\ref{bieta}), given $\t_k$  the random time
$\t_{k+1}-\t_k$ is larger that $a\geq 0$ with probability bounded
from below by $e^{-c\,a}$. Weaker sufficient conditions are also
possible.

 Due to the  Perron--Frobenius Theorem, assumption (A2) implies that the
 time homogeneous Markov chain with generator $L_c (x,t)$ admits a
 unique invariant measure $\mu(\cdot |x,t)$ to which the Markov chain converges
as time goes to $\infty$. If (A2) is not satisfied then it is simple
to exhibit different invariant measures, hence (A2) is equivalent to
the existence of a unique invariant measure. Always due to
Perron--Frobenius Theorem, $\mu (\s|x,t)>0$ for each $\s \in \G$. We
call $\mu(\cdot|x,t)$ the {\sl  quasistationary measure} of the
chemical evolution (frozen in $x,t$).

We observe that  our assumptions   imply the following fact: for
each compact $\cK \subset \bbR^d$, there exists $\k>0$ such that
\begin{align}
& |r(\s,\s'|x,t)-r(\s,\s'|y,t)|\leq \k |x-y|\,,\label{lipr}\\
& | \g_\s(x,t)-\g_\s (y,t) |\leq \k |x-y|\,, \label{lipgamma}\\
 & |\mu(\s|x,t)-\mu (\s|y,t)|\leq \k |x-y|\,, \label{lipmu}
\end{align}
for all $x,y \in \cK$, $t \in [0,T]$, $\s,\s'\in \G$. Indeed, \eqref{lipr} and \eqref{lipgamma} trivially
follow from assumption (A3), while  \eqref{lipmu} can be derived from assumptions (A2) and (A3) (see Appendix \ref{misciotto}, where it is proven that
$\mu(\s|x,t)$ is
$C^1$ in $x,t$).

 Finally, let us point out some consequences  of assumption (A4), some of which will be useful later. To this aim we introduce the averaged
 vector field $\bar F: \bbR^d \times [0,T] \rightarrow \bbR^d$ defined as
\begin{equation}\label{medio}
\bar F (x,t):= \sum _{\s \in \G} \mu(\s|x,t) F_\s (x,t)\,.
\end{equation}
It is simple to check that the force field $\bar F$ satisfies assumption (A4). As consequence, one can easily prove the following result
(see  Appendix  \ref{misciotto}):
\begin{Le}\label{sibillino}
Let $f:[0,T]\rightarrow \bbR$ be  a continuous function. Then,
given $s \in [0,T]$ and $x_0 \in \bbR^d$, the Cauchy problems
\begin{equation}\label{saratoga0}
\begin{cases}
\dot{x}(t)= F_\s \bigl( x(t),t\bigr)f(t)\,, \\
x(s) = x_0\,,
\end{cases}
\end{equation}
and
\begin{equation}\label{saratoga3}
\begin{cases}
\dot{x}(t)= \bar{F}(x(t),t) \,, \\
x(s) = x_0\,,
\end{cases}
\end{equation}
have unique  solutions in the time interval $[s,T]$.
\end{Le}
  Another consequence of  assumption (A4) is described in
 Lemma \ref{phelps} in the Appendix. The reader can check that all our proofs
 work in general for   continuous
 force  fields $F_\s: \bbR^d \times [0,T] \rightarrow \bbR^d$ that are  locally Lipschitz w.r.t.
 $x$, such that \eqref{saratoga0} and \eqref{saratoga3} have unique solutions,
and such that the content of Lemma \ref{phelps} remains valid. Finally, we note that (A4) is satisfied
whenever each $F_\s$ is continuous in $(x,t)$ and Lipschitz in $x$ uniformly in $t$, namely \eqref{lipF}
is satisfied with $\cK=\bbR^d$.

\smallskip

When $|\G|=2$  the above  PDMP can be used to model  motor proteins
with two main chemical states, for example  {\sl detached} and {\sl
attached state} respectively.  In order to shorten the notation, we
denote by $0$ the detached state and by $1$ the attached state,
hence $\G=\{0,1\}$. Note that it must be $p (0,1|x,t)=p(1,0|x,t)=1$,
hence
the generator $L_t$ at time $t$ becomes
\begin{multline*}  L_t f (x, \s):=(1-\s) F_0 (x,t)\cdot  \nabla _x f
(x,0)+\s F_1 (x,t) \cdot \nabla
_x f(x,1)+ \\
(1-\s) \g_0(x,t) \left[ f(x,1)- f(x,0)\right] + \s \g_1 (x,t) \left[
f(x,0)-f(x,1) \right]\,.
\end{multline*}
In this case, the measure \begin{equation}\label{patroclo}
\mu(0|x,t) := \frac{\g_1 (x,t)}{\g_0(x,t)+\g_1(x,t)} \, \qquad
\mu(1|x,t):= \frac{\g_0(x,t)}{\g_0 (x,t)+\g_1(x,t)}
\end{equation}
is the only  invariant measure $\mu(\cdot |x,t)$ of the  Markov
chain on $\G$ with time--independent  generator $L_c(x,t)$.
Moreover, it is also reversible.

\bigskip

We are first interested in analyzing the limiting behavior of PDMPs
where the  time scale  of the chemical transitions is much smaller
than the time scale of  mechanical relaxation, i.e. $\s$ is a fast
variable and $x$ is a slow variable. To this aim, we introduce the
parameter $\l>0$ and  study the evolution on $[0,T]$ of the PDMP
starting in $(x_0, \s_0)$ with characteristics $\left( F, p,
\lambda\g \right)$ as $\l \uparrow \infty$. We call $P^\l_{x_0,
\s_0}$ its law and write $E^\l_{x_0, \s_0}$ for the associated
expectation. For this model we can state a law of large numbers
corresponding to the averaging principle  and a large deviation
principle.


\subsection{Averaging principle}

\begin{Th}\label{LLN} 
Given $(x_0,\s_0)\in \bbR^d\times \G$, call $x_*(t)$ the unique
solution of the Cauchy problem
\begin{equation}\label{trieste}
\begin{cases}
\dot{x}_* (t) = \bar{F}
(x_*(t),t)  \,, \;\; t \in [0,T]\,, \\
x_*(0) =x_0\,,
\end{cases}
\end{equation}
where the averaged force field $\bar F$ is defined as in
\eqref{medio}.

 Then, for any $\d>0$, $\s\in \G$ and for any
continuous function $f : [0,T]\rightarrow \bbR $ it holds
\begin{align}
& \lim_{\l \uparrow \infty} P ^\l_{x_0,\s_0} \left( \left | \int_0^T
f (t) \chi (\s(t) =\s)   dt - \int _0 ^T f(t) \mu( \s|x_*(t),t)  dt
\right |>\d \right)=0\,,\label{lim2} \\
& \lim _{\l \uparrow \infty} P ^\l_{x_0, \s_0} \left( \sup_{0\leq t
\leq T} |x(t)-x_*(t)|>\d\right) = 0\,.  \label{lim1}
\end{align}
\end{Th}
Note that the  PDMP with law $P
^\l_{x_0, \s_0}$ has paths in the Shorohod
 space $D([0,T], \bbR^d \times \G )$ but, as clear from the above
 statement, in order to describe the asymptotic behavior as $\l
 \uparrow \infty$ it is necessary to think the mechanical evolution
 $x(t)$ up to time $T$ as an element of the space $C[0,T]$ endowed with the uniform
 norm, and to identify  the chemical evolution $\s(t)$ to the
 measure--valued vector
 $\r(t)dt \in \cM[0,T]^\G$, where $\r_\s(t)= \chi (\s(t)=\s)$ and   $\cM[0,T]$  denotes the space of
finite nonnegative  Borel measures on
 $[0,T]$ endowed of the weak  topology. This means that
$\mu_n \rightarrow \mu$ if and only if $\mu _n (f) \rightarrow \mu
(f)$ for each $f\in C[0,T]$, where $\mu(f):= \int_0 ^T f(y) \mu (dy
)$ for a generic measure $\mu$. Having in mind the above topologies,
the law of large numbers states that the mechanochemical evolution
$\bigl( x(t), \r(t)dt \bigr)$ converges in probability to $\bigl(
x_*(t), \mu( \cdot|x_*(t),t) dt \bigr)$ as elements of the space $
C[0,T]\times \cM [0,T]^\G$.

\smallskip

\subsection{Large deviation principle}
In order to state the large deviation principle for the above PDMP,
it is convenient to isolate a special subset of $C[0,T]\times
\cM[0,T]^\G$. To this aim we introduce the set  $L[0,T]$ of Lebesgue
measurable functions $f:[0,T]\rightarrow [0,1]$, identified up to
subsets of zero Lebesgue measure. Then we define $\Upsilon $ as
\begin{multline}\label{Uil}
\Upsilon :=\Bigl \{ (x,\rho) \in C[0,T]\times L[0,T]^\G\,:\, \sum
_{\s \in \G} \rho _\s (\cdot)= 1\text{ a.e. }\,, \\ x(t)=x_0
+\sum_{\s \in \G} \int_0 ^t  F_\s (x(s),s) \rho_\s (s) ds ,\;\;
\forall t \in [0,T] \Bigr\}\,,
\end{multline}
where $x_0$ is the starting mechanical state of the system (recall
that the system starts in a deterministic   state $(x_0, \s _0)$ at
time $0$).

The set $\Upsilon $ has to  be thought of  as topological subspace
of $C[0,T]\times\cM[0,T]^\G$  via the identification
$$\bigl(x(t),
\rho _\s(t) \bigr)_{t\in [0,T], \s \in \G }\rightarrow \bigl( x(t) ,
\rho_\s (t) dt \bigr)_{t\in [0,T], \s \in \G}\,.$$
 It can be proved
(see Lemma \ref{annozero} in the Appendix) that $\Upsilon$ is a
compact subspace of $C[0,T]\times \cM[0,T]^\G$, and its topology can
be derived from the metric $d$ defined as
\begin{equation}\label{sara}
d\left( \,(x, \rho)\,,\,(\bar x,\bar  \rho ) \,\right) = \|x-\bar x
\|_\infty +\sum _{\s \in \G} \Bigl( \sup _{0\leq t \leq T}
\Bigl|\int _0 ^t \bigl[ \rho_\s(s) -\bar \rho_\s (s) \bigr]ds
\Bigr|\Bigr)\,.
\end{equation}
It is clear that for each stochastic evolution $\bigl(x(t),
\s(t)\bigr)_{t \in [0,T]}$ of the system, the path
\begin{equation}\label{solarium}
\bigl(x(t) ,
\chi (\s(t)=\s) \bigr)_{t\in [0,T], \s \in \G}
\end{equation} is an element of
$\Upsilon$. In what follows we call $Q^\l_{x_0, \s_0} $ the law on
$\Upsilon$ of the random path (\ref{solarium}), when $\bigl(x(t), \s(t)\bigr)_{t \in
[0,T]}$ is chosen with law $P^\l_{x_0,\s_0}$.

\smallskip
Before stating our second main result we introduce some notation: we
set
\begin{equation}\label{wc}
W:=\{ (\s,\s')\in \G\times \G\,:\, \s \not = \s'\}\,,
\end{equation}
and, given a nonnegative measure $\p$ on $\G$ and nonnegative
numbers $r(\s,\s')$, with $(\s,\s')\in W$, we define the function
$j(\p, r)$  as
\begin{equation}\label{pranzo}
j(\p, r):=\sup _{z\in (0,\infty)^\G} \sum _{(\s.\s')\in W } \p(\s)
r(\s,\s')\left [ 1-\frac{z_{\s'}}{z_\s } \right] \,,
\end{equation}
i.e.
\begin{equation}\label{pranzobis}
j(\p,r) := \sup _{z \in (0, \infty)^\G} -\sum _{\s \in \G} \p(\s)
\frac{ (\cL z)_\s }{z_\s}
\end{equation}
where $\cL$ is the generator of the continuous--time Markov chain on
$\G$ jumping from $\s$ to $\s'\not=\s$ with rate $r(\s,\s')$. We
can finally state our large deviation principle (LDP):

\begin{Th} \label{LDP} 
Given $(x_0,\s_0 )\in \bbR^d \times \G$,
  the family of probability measures $Q^\l_{x_0, \s_0}$ on $\Upsilon$
satisfies a LDP with parameter $\l$ and with rate function
$J:\Upsilon \rightarrow [0,\infty)$ defined as
\begin{equation}\label{variazione1}
J  ( x, \rho )= \int _0 ^T j \bigl( \rho(t), r(\cdot,\cdot|x(t),t)
\bigr) dt
\end{equation}
where $j$ has been defined in \eqref{pranzo}, \eqref{pranzobis}.

 Moreover, if the quasistationary measures $\mu(\cdot|x,t)$
are reversible for the chemical generators $L_c \bigl(x,t\bigr)$,
then
\begin{equation}\label{variazione2}
J (x, \rho) =-\int_0 ^T  \left< \sqrt{    \frac{\rho_\cdot  (t)
}{\mu(\cdot |x(t),t)}}, L_c \bigl( x(t),t \bigr )   \sqrt{
\frac{\rho_\cdot (t) }{\mu(\cdot |x(t),t)}  } \right>_t dt
\end{equation}
where $<\cdot,\cdot >_t$ denotes the scalar product in $L^2\left(\G,
\mu(\cdot|x(t),t)\right)$.
\end{Th}

We recall that the above LDP means that
\begin{align}
& \limsup _{\l \uparrow \infty} {1 \over \l} \log Q^\l_{x_0, \s_0}
(C) \leq - J(C)\,,
\qquad \forall C\subset \Upsilon \text{ closed}\,,\label{spiegelupper}\\
&  \liminf_{\l \uparrow \infty} {1 \over \l}\log Q^\l_{x_0,\s_0} (O)
\geq - J(O)\, , \qquad \;\forall O \subset \Upsilon \text{ open}\,,
\label{spiegeldown}
\end{align}
where
$$ J(S)= \inf _{(x,\rho)  \in S} J(x,\rho)\,, \qquad S \subset
\Upsilon\,.
$$
Moreover, the function $J$ must be a lower semi--continuous function
such that $J \not \equiv \infty$.

\smallskip We point out that $j\bigl( \,\rho(t), r(\cdot,\cdot|x(t),t)\, \bigr) $ corresponds to the
large deviation functional in $\rho(t) $ of the empirical measure
associated to the time--homogeneous and continuous--time Markov
chain on $\G$ jumping from $\s$ to another chemical state  $\s'$
with transition rate $r(\s, \s' |x(t),t)$, $t$ being thought of as a
fixed parameter here \cite{dH}. As the reader will observe, this LDP
for time--homogeneous Markov chains will be one of the main
ingredients in our proof of Theorem \ref{LLN}. Other details on the
variational problem (\ref{pranzo}) will be given in Section
\ref{var_problema}.

In the case of two chemical states, since we know the form of the
quasistationary measure (see (\ref{patroclo})) and that it is
reversible w.r.t. the  chemical generator $L_c$, the above theorem
implies:

\begin{Cor}
In the case of two chemical states,  $\G=\{0,1\}$, the family of
probability measures $Q^\l_{x_0,\s_0}$ on $\Upsilon$  satisfies a
large deviation principle with parameter $\l$ and with rate function
$J:\Upsilon \rightarrow [0,\infty]$ given by
\begin{equation}\label{variazione4}
J(x,\rho)=\int_0^T
\left[\sqrt{\rho_0(t)\g_0\left(x(t),t\right)}-\sqrt{\rho_1(t) \g_1
\left(x(t),t\right)}\right]^2dt\,.
\end{equation}
\end{Cor}

\smallskip

\subsection{Coarse--grained process}\label{luna}

We finally consider  the PDMP with fast and slow chemical jumps and
study  the asymptotic  behavior of the  coarse--grained process
obtained from the original one by keeping knowledge of the
mechanical state and only of the chemical  metastate of the system.
More precisely, we consider a partition of $\G$, $\Gamma= \G_1 \cup
\G_2\cup \dots \cup \G_\ell$ where $|\G_i|\geq 1$.  We rename the
elements of $\G$ by calling $\s_{i,1}, \s_{i,2}, \dots , \s_{i,
n_i}$ the elements of $\G_i$. The probability $P^\l_{x_0,\s_0}$ is
now the law of the PDMP starting in $(x_0, \s_0)$ with generator
\begin{equation}\label{generarexxx}
L_t g (x,\s) = F_\s (x,t) \cdot \nabla_x g(x, \s) +
 \sum _{\s' \in \G: \s'\not = \s} \l(\s,\s') r(\s,\s'|x,t) \bigl (
g (x,\s') - g(x, \s) \bigr)\,,
\end{equation}
where
$$ \l (\s,\s'):=
\begin{cases}
\l & \text{ if } \s, \s' \in \G_i \text{ for some $i$}\,, \\
1 & \text{ otherwise}\,.
\end{cases}
$$
Note that  the chemical jumps between states in the same chemical
class $\G_i$ take place in a short time of order $O(1/\l)$, while
chemical jumps between states in different chemical classes take
place in times of order $O(1)$. Hence, it is natural to call
 the  classes
$\G_i$ {\sl chemical  metastastes}. Below, we denote by
$$ \G^{(\ell)}:= \{1,2,\dots,\ell\}
$$
the family of metastates $\G_i$, writing $i$ for the metastate $\G_i$.
 When the
rates $r(\s,\s'|x, t)$ do not depend on the mechanical state $x$,
the chemical evolution is determined by a time--inhomogeneous Markov
chain on $\G$ with strong and weak interactions and this situation
has been studied in detail \cite{YZ}. We consider here the
fully--coupled case where the transition rates of the chemical jumps
depend on the mechanical state. As discussed in \cite{YZ}[Chapter
7], it is simple to give examples where  $P^\l _{x_0,\s_0}$ does not
converge weakly as $\l \uparrow \infty$ since the family $ \{P^\l
_{x_0, \s_0}\,:\, \l>0\}$ is not tight. Nevertheless, one can obtain
from $P^\l_{x_0,\s_0}$ a $\l$--dependent coarse--grained process
weakly converging to a new PDMP. In order to describe precisely this
result it is convenient to fix some notation. We can write the
chemical generator $L_c (x,t)$ as
\begin{equation}
L_c (x,t) = \widehat L _c (x,t)  + \l \widetilde L_c (x,t) \,,
\end{equation} where $\widehat L _c (x,t)$ and  $ \widetilde L_c (x,t)$ are
both $\l$--independent Markov generators on  $\G$ parametrized by
$x$ and $t$. Trivially, $\widetilde L_c (x,t)$ has a diagonal--block
form:
$$
\widetilde L_c (x,t)= \left(
\begin{array}{cccc}
\widetilde L_c^1 (x,t)   &   &  & \\
 & \widetilde L_c ^2 (x,t)  & &\\
       &                      &  \ddots & \\
       &                       &        &  \widetilde L ^\ell _c (x,t)
\end{array}
\right)\,,
$$
where $\widetilde L^i _c (x,t)$ is a Markov generator on $\G_i$.

\smallskip

   We can now state our assumptions. We keep our previous
assumptions (A1), (A3) and (A4), while we replace  assumption (A2) with
the following  (A2'):

\begin{itemize}

\item Assumption (A2'):  for each $i\in \G^{(\ell)}  $ and $(x,t ) \in \bbR^d \times [0,T]$,  the generator $\widetilde L_c ^i (x,t)$ is
irreducible on $\G_i$.

\end{itemize}
Motivated by assumption (A2'),  we write $\mu_i (\cdot |x,t)$ for
the unique invariant  probability  measure on $\G_i$, i.e. $\mu_i
(\cdot|x,t)$ is the quasistationary measure for the generator
$\widetilde L_c ^i (x,t)$.
Given $i \in \G^{(\ell)}$ and $(x,t)\in \bbR^d \times [0,T]$ we
define the vector field
\begin{equation}\label{babbol}
 F_i (x,t):= \sum _{\s \in \G_i}
F_\s(x,t) \mu _i (\s |x,t)\,.
\end{equation}

 Given a state $\s \in \G$ we define  $\a (\s)$ as
$\a(\s)=i$ if $\s \in \G_i$, moreover we denote by $R_{x_0, \s_0}
^\l$ the law on $D\bigl([0,T], \bbR^d\times  \G^{(\ell)} \bigr)$
obtained as image of $P_{x_0, \s_0}^\l$ under the map
$(x(t), \s(t)) \rightarrow \bigl(\,x(t), \alpha ( \s(t) ) \, \bigr)
$  (we will often write $\a(t)$ in place of $\a(\s(t))$). $R_{x_0,
\s_0} ^\l$ is the law of the above mentioned coarse--grained
process. Note that in general this process is not Markovian, however
it converges to a PDMP:
\begin{Th}\label{LLNws}
Given $T>0$, as $\l $ goes to $\infty$ the law $R_{x_0, \s_0} ^\l$
weakly converges to the law of the  PDMP on $\bbR^d \times
\G^{(\ell)}$ with generator
\begin{equation}\label{bigbang}
L_s g (x,i ) = F_i  (x,s) \cdot \nabla_x g(x, i) +
 \sum _{j \not = i } r (i,j |x,s) \bigl (
g (x,j) - g(x, i) \bigr)\,,
\end{equation}
where, given $1\leq i \not = j \leq \ell$,
\begin{equation}\label{teramane}
 r(i,j|x,s):= \sum _{\s\in
\G_i} \sum _{\s' \in \G_j} \mu_i (\s|x,s) r(\s,\s'|x,s)\,.
\end{equation}
\end{Th}
Note that, due to assumptions (A1) and (A4), both the chemical and
mechanical evolutions for the original PDMP and the limiting one are
well defined.

\bigskip

\subsection{Outline of the paper}
 The remaining part of the paper is organized as follows: in
 Section \ref{motori} we discuss some biological applications,
in Section \ref{var_problema} we consider  the variational problem
(\ref{pranzo}): we  show that the r.h.s. of (\ref{variazione1})
equals the r.h.s. of (\ref{variazione2}) if $\mu(\cdot|x(s),s)$ is
reversible w.r.t. $L_c \bigl(x(s),s\bigr)$ and we show other results
useful for the proof of the LDP, in Section \ref{dim_LLN} we prove
the LLN stated in Theorem \ref{LLN}, in Section \ref{dim_LD} we
prove the LDP stated in Theorem \ref{LDP} and in Section
\ref{paccolone} we prove the asymptotic behavior of the
coarse--grained process described in Theorem \ref{LLNws}. Finally,
in the Appendix we prove some technical results used in the paper.

\section{Some biological applications}\label{motori}

In this section we discuss an application of the above averaging
principles for PDMPs to molecular motors. Further results will be
presented in a companion paper \cite{FGR}.

\smallskip

   As  already said, molecular motors (MMs)
     are proteins working as engines on the nanometer scale, they generate forces of  piconewton order
 and are usually powered by the chemical energy derived from ATP hydrolysis. Most of them
are linear motors, i.e. they  proceed in a given direction along some filament, which has a function similar to a railways track. We refer the interested reader
  to \cite{H}.

 In the power stroke picture,  force is generated by the swinging motion of some part of the protein (the  {\sl lever--arm}).
   In order to study the working of the MM, its  state can  be described  simply by the pair $(x,\sigma)$,
   where $x$ is a continuum variable specifying  the configuration of the lever--arm, while
 $\sigma$ is a discrete variable describing  the chemical state of the MM (i.e.   bound to or detached from the track filament,
  bound to ATP or to the hydrolysis products).  Given the chemical state $\s$, the evolution of the mechanical state is
  determined by Newtonian laws as
  \begin{equation}\label{meccanica}
   m \ddot{x}= -\g \dot{x} + F_\s (x,t) + \xi (t) \,,
   \end{equation}
where $\g \dot{x}$ denotes the friction force, $F_\s (x,t)$ is the
force field defined as $F_\s (x,t)= - \nabla_x U_\s (x,t)$, $U_\s
(x,t)$ being the free energy, and $\xi(t)$ denotes the thermal force
due to the environment. Since the inertia effects are negligible and
the dynamics  is overdamped \cite{H}, one can disregard the term $m
\ddot{x}$. In what follows, we will neglect also the thermal force
$\xi (t)$ as first approximation. Indeed, contrary to the  modeling
of MMs as Brownian ratchets, in the power stroke picture the thermal
force  $\xi(t)$ is not relevant for many qualitatively aspects.
Taking $\g=1$ without loss of generality, (\ref{meccanica}) reduced
to equation (\ref{sole}).
 Despite our approximation, the thermal fluctuations remain  essential in the chemical kinetics, since the chemical
jumps are stochastic and must satisfied the detailed balance equation:
\begin{equation}\label{dbe} \frac{ r(\s,\s'|x,t)}{ r(\s',\s |x,t)} = \exp \left\{ - \beta
\bigl[U_{\s'} (x,t)- U_{\s} (x,t)\bigr] \right\}\,,\qquad \forall
\s, \s' \in\G\,,
\end{equation}
where $\b$ is the inverse temperature and, as before,
 $r(\s,\s'|x,t)$ denotes
  the probability rate for a jump from $\s$ to another state $\s'$. Due to the above considerations, the mechanochemical evolution of the MM
  can be described by a PDMP satisfying the detailed balance equation \eqref{dbe}.
  Since  there is experimental evidence that some jump rates must depend on $x$,
  the PDMP is fully--coupled (see \cite{D} and references therein).

In this section we show how
 instabilities in the response to external solicitations  of MMs follow from the multiscale character of the system by means of the averaging principle
 stated  in Theorem \ref{LLNws}. Let us first explain what one means by instabilities (see \cite{J},\cite{VD} for some examples).
In ordinary conditions a MM moves typically in a given direction along the filament, at a given averaged
 speed $v$ depending on environmental parameters as temperature and ATP concentration. If some external force opposes to the motion,
  the MM slows down and eventually stops at the stall force $f_s$.
The observed instability is of the following kind: in some range of the values of the environmental parameters the MM does not stall,
 rather for force values close to $f_s$  it can proceed in both directions along the filament.

\smallskip



Inspired by \cite{VD}  we study this phenomenon by means of a  PDMP
with three chemical states and  one dimensional
 mechanical variable so that the full state is  described by $(x,\s)\in \bbR\times\Gamma$ where $\Gamma=\{0,1,2\}$.
For the interpretation of the different chemical states we refer to \cite{VD}, we only mention that when $\s=1,2$ the MM is attached to the filament (but
in different ways), while in state $\s=0$ the MM is detached. The force fields are defined as
\begin{equation}
\begin{cases}\label{sante}
F_0(x)=- x\\
F_1(x)=- x-f\\
F_2(x)=-(x-1)-f\,.
\end{cases}
\end{equation}
This is the most simple choice to give the mechanical variable an
equilibrium position in each state. Here $f$ is a control parameter,
representing the external force exerted on the filament, and
therefore on the MM when bound to the filament. These one
dimensional force fields admit three convex potential functions:
\begin{equation}\label{ra}
\begin{cases}
U_0(x)=\frac{1}{2} x^2\\
U_1(x)= \frac{1}{2}x^2+fx\\
U_2(x)= \frac12(x-1)^2+fx+\epsilon
\end{cases}
\end{equation}
such that $-\partial_x U_\s(x)=F_\s(x)$. Here $\epsilon$ is a second
control parameter, which can be related to the ATP concentration.
Finally, the  transition rates $r(\s,\s'|x)$  must satisfy the
detailed balance equation \eqref{dbe} and assumptions (A1)  and (A3).
 There is experimental evidence that the jumps between states $1$ and $2$ are much faster than the other chemical jumps
 and it is reasonable to suppose that
the jumps between states $1$ and $2$ are faster
than   the mechanical evolution (see \cite{D}, \cite{VD} and
references therein). Hence, $\G$ can be partitioned in two chemical
metastates $\{0\}$ and $\{1,2\}$, while the rates $r(1,2|x)$ and
$r(2,1|x)$ can be rescaled by a factor $\l$. We assume these rates to be positive, thus implying assumption (A2') of
Section \ref{luna}.
 Note that the rescaling
does not alter the validity of the detailed balance equation
\eqref{dbe}. By Theorem \ref{LLNws}, as $\l \uparrow \infty$, the
coarse--grained
 PDMP converges weakly to a  new PDMP with state space $\bbR  \times \{0,*\}$ ($0$ refers to the metastate $\{0\}$ and $*$ refers to the
 metastate $\{1,2\}$). The  force field in the metastate $\{0\}$ coincides with $F_0$, while in the metastate $*$ is given by $F_*$ defined as
\begin{equation}
 F_*(x)=\frac{e^{\b \Delta U(x)}}{1+e^{\b \Delta
U(x)}}F_1(x)+\frac{1}{1+e^{\b \Delta U(x)}}F_2(x)\,,
\end{equation}
where $\Delta U(x)=U_2(x)-U_1(x)$. Note that assumptions (A2') and (A4') are fulfilled.
It must be noted that while the force fields in \eqref{sante} have a
single equilibrium  point, $F_\s(x)=0$, the force field $ F_*(x)$
may have more than one depending on the value of the control
parameters $\epsilon,\beta$ and  $ f$. To this aim, we proceed as
follows:
First we observe that
$$ F_*(x)
=F_1(x)+\frac{1}{1+e^{\beta\Delta U(x)}}\left(F_2(x)-F_1(x)\right)
=-x-f+\frac{1}{1+e^{\beta\Delta U(x)}}\,.
$$
Let us take $f=0$. Since $\D U(x)=U_2(x)-U_1 (x)= -x+1/2+\e$, we obtain
\begin{equation*}
\begin{split}
\partial_x F_*(x)= - 1+\frac{\b y}{(1+y)^2}\,, \qquad y:=
e^{\beta(\frac 12-x+\epsilon)}\,.
\end{split}
\end{equation*}
Let us analyze the
 region of positive slope $$I_+=\{ x \in \bbR\,:\, \partial _x  F_* (x) >0\}= \{ x\,: \b y(x)>(1+y(x))^2\}\,.$$
Since $y>0$, $I_+= \emptyset $ if $\b\leq 4$. If $\b>4$,  then $
I_+$ is the finite interval $(a_-,a_+)$ such that
$$ y(a_\pm )=\frac{\beta-2\mp \sqrt{\beta^2-4\beta}}{2}\,.$$
Note that the r.h.s. is always positive. Since $y(0)=e^{\b
(\frac{1}{2}+\e)}$, $0 \in I_+$ if $\b>4$ and
\begin{equation}\label{blame}
-\frac12+\frac{1}\beta \log\left( \frac{\beta-2-\sqrt{\beta^2-4\beta}}2\right)<\epsilon<-\frac12+\frac{1}\beta
 \log\left( \frac{\beta-2+\sqrt{\beta^2-4\beta}}2\right).
\end{equation}

In conclusion, if $f=0$, $\b>4$ and (\ref{blame}) is satisfied, then
 $\partial _x  F_*(x)>0$ only in a given interval $I_+$ containing the origin, while $\lim _{x \rightarrow \pm \infty} F_* (x)= \mp \infty$.
 Since adding $f$ has the only effect to translate the graph of $F_*$ along the ordinate axis, we
conclude that for a suitable value of $f$ the equation $ F_*(x)=0$
has three solutions $x_-< 0<x_+$ and that
$$  F_* (x) \begin{cases}
>0 & \text{ if } x< x_-\,,\\
<0 & \text{ if }x_- < x< 0\,,\\
>0 & \text{ if } 0<x<x_+ \,,\\
<0 & \text{ if } x>x_+\,.
\end{cases}
$$ See Figure \ref{rino} below.
 Now it is simple to check that if the new PDMP  starts in the
mechanical state $x_0<0$  [$x_0>0$],
 then $x(t)$ eventually enters the absorbing
interval $(x_-,0)$ [$(0,x_+)$]. Due to the description of the power
stroke mechanism (see \cite{FGR}), this implies that typically the
motor moves towards left in the first case and towards  right in the
latter. Hence, for  suitable parameters $\e$ and $ \b$ associated to
ATP concentration and temperature and for suitable external forces
$f$ the MM can move  in both directions of the filament, depending
on the initial configuration of the lever--arm. This behavior is a
form of instability in the response of the MM.  We point out that
for the original PDMP with three  states,
 there is only one absorbing interval given by $(0,1)$  which $x(t)$ eventually enters a.s.
 Hence,
 the instability is related to the ergodicity breaking  of the system.

\begingroup
  \makeatletter
  \providecommand\color[2][]{%
    \GenericError{(gnuplot) \space\space\space\@spaces}{%
      Package color not loaded in conjunction with
      terminal option `colourtext'%
    }{See the gnuplot documentation for explanation.%
    }{Either use 'blacktext' in gnuplot or load the package
      color.sty in LaTeX.}%
    \renewcommand\color[2][]{}%
  }%
  \providecommand\includegraphics[2][]{%
    \GenericError{(gnuplot) \space\space\space\@spaces}{%
      Package graphicx or graphics not loaded%
    }{See the gnuplot documentation for explanation.%
    }{The gnuplot epslatex terminal needs graphicx.sty or graphics.sty.}%
    \renewcommand\includegraphics[2][]{}%
  }%
  \providecommand\rotatebox[2]{#2}%
  \@ifundefined{ifGPcolor}{%
    \newif\ifGPcolor
    \GPcolorfalse
  }{}%
  \@ifundefined{ifGPblacktext}{%
    \newif\ifGPblacktext
    \GPblacktexttrue
  }{}%
  \let\gplgaddtomacro\g@addto@macro
  \gdef\gplbacktext{}%
  \gdef\gplfronttext{}%
  \makeatother
  \ifGPblacktext
    \def\colorrgb#1{}%
    \def\colorgray#1{}%
  \else
    \ifGPcolor
      \def\colorrgb#1{\color[rgb]{#1}}%
      \def\colorgray#1{\color[gray]{#1}}%
      \expandafter\def\csname LTw\endcsname{\color{white}}%
      \expandafter\def\csname LTb\endcsname{\color{black}}%
      \expandafter\def\csname LTa\endcsname{\color{black}}%
      \expandafter\def\csname LT0\endcsname{\color[rgb]{1,0,0}}%
      \expandafter\def\csname LT1\endcsname{\color[rgb]{0,1,0}}%
      \expandafter\def\csname LT2\endcsname{\color[rgb]{0,0,1}}%
      \expandafter\def\csname LT3\endcsname{\color[rgb]{1,0,1}}%
      \expandafter\def\csname LT4\endcsname{\color[rgb]{0,1,1}}%
      \expandafter\def\csname LT5\endcsname{\color[rgb]{1,1,0}}%
      \expandafter\def\csname LT6\endcsname{\color[rgb]{0,0,0}}%
      \expandafter\def\csname LT7\endcsname{\color[rgb]{1,0.3,0}}%
      \expandafter\def\csname LT8\endcsname{\color[rgb]{0.5,0.5,0.5}}%
    \else
      \def\colorrgb#1{\color{black}}%
      \def\colorgray#1{\color[gray]{#1}}%
      \expandafter\def\csname LTw\endcsname{\color{white}}%
      \expandafter\def\csname LTb\endcsname{\color{black}}%
      \expandafter\def\csname LTa\endcsname{\color{black}}%
      \expandafter\def\csname LT0\endcsname{\color{black}}%
      \expandafter\def\csname LT1\endcsname{\color{black}}%
      \expandafter\def\csname LT2\endcsname{\color{black}}%
      \expandafter\def\csname LT3\endcsname{\color{black}}%
      \expandafter\def\csname LT4\endcsname{\color{black}}%
      \expandafter\def\csname LT5\endcsname{\color{black}}%
      \expandafter\def\csname LT6\endcsname{\color{black}}%
      \expandafter\def\csname LT7\endcsname{\color{black}}%
      \expandafter\def\csname LT8\endcsname{\color{black}}%
    \fi
  \fi
  \setlength{\unitlength}{0.0500bp}%
  \begin{picture}(7200.00,5040.00)%
    \gplgaddtomacro\gplbacktext{%
      \csname LTb\endcsname%
      \put(990,1072){\makebox(0,0)[r]{\strut{}-0.4}}%
      \put(990,1895){\makebox(0,0)[r]{\strut{}-0.2}}%
      \put(990,2718){\makebox(0,0)[r]{\strut{} 0}}%
      \put(990,3541){\makebox(0,0)[r]{\strut{} 0.2}}%
      \put(990,4364){\makebox(0,0)[r]{\strut{} 0.4}}%
      \put(1122,440){\makebox(0,0){\strut{}-1}}%
      \put(2548,440){\makebox(0,0){\strut{}-0.5}}%
      \put(3974,440){\makebox(0,0){\strut{} 0}}%
      \put(5400,440){\makebox(0,0){\strut{} 0.5}}%
      \put(6826,440){\makebox(0,0){\strut{} 1}}%
      \csname LTb\endcsname%
      \put(220,2718){\rotatebox{90}{\makebox(0,0){\strut{}$F_*(x)$}}}%
      \put(3974,110){\makebox(0,0){\strut{}$x$}}%
    }%
    \gplgaddtomacro\gplfronttext{%
    }%
    \gplbacktext
    \put(0,0){\includegraphics{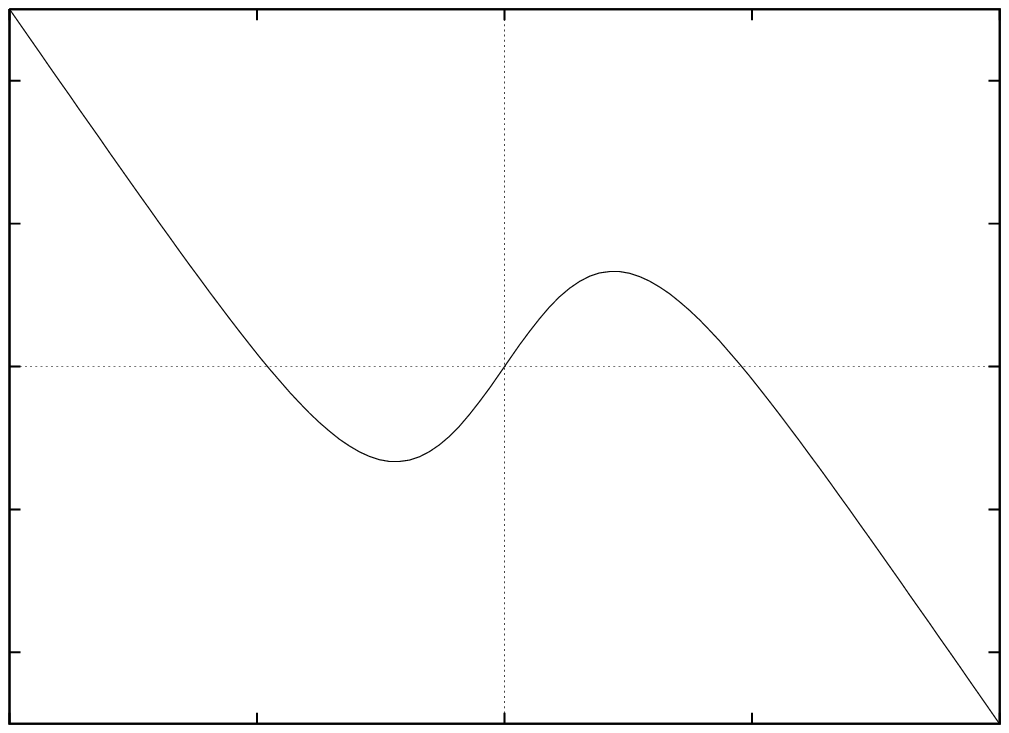}} \label{rino}%
    \gplfronttext
  \end{picture}%
\endgroup

\section{The variational problem (\ref{pranzo})}\label{var_problema}

In this section, we briefly analyze the variational problem
(\ref{pranzo}). Part of the  content of Lemma \ref{pierbau} below is
well known, nevertheless we recall its  derivation since we need
some additional developments in order to prove the LDP.

We call $\mathcal S$ the subset of $[0,\infty)^W$ given by the
elements $c\in [0, \infty)^W$ satisfying the following
irreducibility condition: given $\s \not = \s'$ in $\G$ there exists
a finite sequence $\s_1,\s_2, \dots, \s_n$ such that $\s_1=\s$,
$\s_n=\s'$ and $c(\s_i , \s_{i+1}) >0$ for all $i=1,\dots, n-1$. We
define $\mathcal{J}:[0,\infty)^W\rightarrow \bbR$ as
\begin{equation}\label{emma28} \mathcal{J} (c) = \sup _{z \in (0,\infty)^\G}
\widehat{\mathcal{J}}(c,z)\,, \qquad
\widehat{\mathcal{J}}(c,z):=\sum _{(\s,\s')\in W} c(\s,\s') \bigl( 1
- \frac{ z_{\s'}}{z_\s} \bigr)\,.
\end{equation}
\begin{Le}\label{pierbau}
The function $\mathcal{J}$  is convex and  continuous, and takes
values in $[0,\infty)$. Moreover, for each  $c\in \mathcal S$,
 the supremum  on $(0,\infty)^\G$
of the function $\widehat{\cJ}(c, \cdot)$ is a maximum and the set
of maximum points is given by  the ray $\{t\widetilde z\,:\, t
>0\}$, where $\widetilde z \in (0,\infty)^\G$ is the unique solution of the system
\begin{equation}\label{frignotta}
\sum_{\s'\in \G} c(\s,\s') \frac{z_{\s'}}{z_\s} = \sum _{\s'\in \G}
c(\s',\s) \frac{z_\s}{z_{\s'}}\,, \qquad \s \in \G\,,
\end{equation}
such that $\sum _{\s \in \G} \widetilde z_\s^2=1$.
\end{Le}

\begin{proof}
Since  $\cJ $  is the supremum of a family of linear functions in
$c\in [0,\infty)^W$ parametrized by $z$, $\mathcal{J}$ is convex and
lower semicontinuous  on  the set $[0,\infty)^W$, which is    a
locally simplicial set (see \cite{R}[Chapter 10]). Hence we can apply Theorem 10.2 in \cite{R}
implying that $\mathcal{J}$ is upper semicontinuous. This concludes
the proof of the continuity of $\mathcal{J}$. Since $ 0=  \widehat \cJ (c,
\underline{1} )\leq \cJ (c) \leq \sum_{(\s,\s')\in W} c(\s,\s')
<\infty,$ we get that $\cJ(c)\in [0,\infty)$.

\smallskip

Let us now assume that $c\in \mathcal S$ and prove the conclusion of the Lemma.
We can write $\cJ(c)= \bigl[ \sum_{(\s,\s')\in W } c(\s,\s') \bigr]- \inf_{z\in
(0, \infty)^\G} \Phi (z) $, where
\begin{equation}\label{lorella08}
\Phi(z):= \sum _{(\s,\s')\in W} c(\s,\s') \frac{z_{\s'}}{z_{\s}}\,,
\qquad z\in (0,\infty)^\G\,.
\end{equation}
Hence it is enough to prove the analogous statements for the infimum
of $\Phi$.  Since $\Phi(z)= \Phi (\g z)$, for all $\g>0$, in order
to study its infimum we can restrict $\Phi$ to the set $\cA:=
\{ z \in (0,\infty)^\G\,:\, \sum _{\s \in \G} z_\s ^2 = 1\}$.
 Trivially $\Phi$ is a positive function.
Let $z^{(n)}\in \mathcal A$ be a minimizing sequence for $\Phi$, i.e. a
sequence such that
$$
\lim_{n\to \infty}\Phi(z^{(n)})=\inf_{z\in
(0, \infty)^\G} \Phi (z)\,.
$$
Due to the fact that $\mathcal A$ is relatively compact in $\bbR^\G$, at cost to take a  subsequence, we
can assume that $z^{(n)}$ is convergent. Let us first  suppose that
$\widetilde{z}=\lim_{n\to \infty}z^{(n)}$ is such that there exists some $\sigma\in \G$
with $\widetilde{z}_\s=0$. Then necessarily also $\widetilde{z}_{\s'}=0$ for any $\s'$
such that $c(\s,\s')>0$. In fact, otherwise, we would have:
$$
\Phi(z^{(n)})\geq c(\s,\s')\frac{z^{(n)}_{\s'}}{z^{(n)}_{\s}}\stackrel{n\to \infty}{\longrightarrow}+\infty\, .
$$
Iterating this argument and  using the fact that $c\in \mathcal S$,
we deduce that necessarily $\widetilde{z}=0$ but this is
incompatible with  the fact that $\widetilde{z}$ belongs to the
closure of $\mathcal A$. Therefore, it must be $\widetilde{z}\in
\mathcal A$. Since the function  $\Phi$ is continuous on $\mathcal
A$, we conclude that
$$
\inf_{z\in
(0, \infty)^\G} \Phi (z) =\lim_{n\to \infty}\Phi(z^{(n)})=\Phi(\widetilde{z})\, ,
$$
and the infimum is a minimum.
Since
$\widetilde z$ is a minimum point for the $C^\infty$ function $\Phi$ on
$(0,\infty)^\G$, it must be $\nabla \Phi (\widetilde z)=0$ . As follows from the  computations below, this
identity coincides with the system of equations (\ref{frignotta}).
%
Finally, we prove that such a system has a unique solution on $\cA$.
Indeed, straightforward computations give that
\begin{align*}
&  \frac{
\partial \Phi }{
\partial z_{\widetilde \s}}(z) =
\sum_{\s:\s\not =\widetilde \s} \left(c(\s,\widetilde \s) \frac{1}{z_\s}-
c(\widetilde \s, \s)
\frac{z_\s}{z_{\widetilde \s}^2}\right)\,,\\
& \frac{ \partial ^2 \Phi}{\partial z_{\widetilde \s} ^2}(z) = \sum
_{\s:\s\not = \widetilde \s} 2 c( \widetilde \s, \s)
\frac{ z_\s }{ z_{\widetilde \s} ^3}\,,\\
& \frac{ \partial ^2 \Phi}{ \partial z _{\widehat \s} \partial z _{\widetilde
\s}} (z)= - c(\widehat \s, \widetilde \s) \frac{ 1}{z_{\widehat \s}^2}- c(\widetilde \s, \widehat \s)
\frac{1}{z_{\widetilde \s} ^2} \,, \qquad \widetilde \s \not =\widehat \s\,.
\end{align*}
Setting
$$ c (\s,  \widetilde \s|z)=
\begin{cases}
c(\s, \widetilde  \s) \frac{ z_{ \widetilde \s } }{ z_\s } \,, &
\text{ if }  \s \not =  \widetilde \s\,,\\
- \sum _{ \s'\,:\,\s'\not = \s } c(\s,  \s') \frac{ z_{\s'}}{z_\s}
\,,  & \text{ if } \s= \widetilde \s\,,
\end{cases}
$$
the above computations imply for each $z \in (0,\infty)^\G$ that
\begin{equation}\label{padova}
\sum _{\s\in \G }\sum_{\widetilde  \s \in \G} a_ \s a _{\widetilde \s
}\frac{
\partial ^2 \Phi}{\partial z_{ \s} \partial z_{ \widetilde \s} }(z) =-2
\sum _{\s\in \G } \sum_{\widetilde \s \in \G}  \frac{ a_ \s}{z_{\s} } c(
\s,\widetilde   \s|z)\frac{ a_{\widetilde\s}}{z_{\widetilde \s}}\,,\qquad a
\in\bbR^\G\,.
\end{equation}
On the other hand, from the system of identities (\ref{frignotta}), it is
 simple to derive that, whenever $\nabla \Phi (z)=0$, it holds
\begin{equation}\label{treviso}
\sum _{\s}\sum_{\widetilde \s} c(\s,\widetilde \s|z) \bigl(b_{\widetilde
\s}-b_{\s} \bigr) ^2 =-2 \sum_{\s}\sum_{\widetilde\s } b_{\s} c(\s,
\widetilde \s|z) b_{\widetilde \s}\,, \qquad b \in \bbR^\G \,.
\end{equation}
Setting $b_\s=a_\s/z_\s$ in \eqref{treviso} and comparing the resulting identity with \eqref{padova}, we obtain that
\begin{equation}\label{venezia}
\sum _{\s }\sum_{\widetilde  \s } a_ \s a _{\widetilde \s
}\frac{
\partial ^2 \Phi}{\partial z_{ \s} \partial z_{ \widetilde \s} }(z) =
\sum _{\s}\sum_{\widetilde \s} c(\s,\widetilde \s|z) \bigl(
\frac{a_{\widetilde \s}}{z_{\widetilde \s} } -\frac{a_\s}{z_\s} \bigr) ^2
\end{equation}
if $\nabla \Phi (z)=0$.
 Hence, this last condition implies that  the r.h.s. of (\ref{venezia}) is
zero
if and only is $a_\s /z_\s= a_{\widetilde \s}/z_{\widetilde \s} $
whenever $c(\s,\widetilde \s)>0$. Due to the fact that $c\in
\mathcal S$, this implies that  the vector $a$ is proportional to
$z$. We observe that the tangent space in $z$ to $\cA$, i.e. $T_z
\cA$, is given by the $\bbR^\G$--vectors orthogonal to $z$.  Hence,
if $z\in \cA$ satisfies $\nabla \Phi (z)=0$ then the map $\Phi_{|A}$
($\Phi$ restricted to $\cA$) has strictly positive defined Hessian
in $z$, thus implying that $z$ is a strict local minimum. On the
other side, if $z \in \cA$ is an extremal point of $\Phi_{|\cA}$
then $\nabla \Phi (z)$ is orthogonal to the tangent space $T_z \cA$,
which is given by all vectors orthogonal to $z$. But, since $\Phi
(tz )=\Phi(z)$ for all $t>0$, by differentiating this equality in
$t=1$ we obtain that $\nabla \Phi (z)\cdot z =0$. In conclusion: the
set of extremal points of $\Phi _{|A}$ coincides with the set  $\{z
\in \cA:\nabla \Phi (z)=0\}$ and we know that all these points are
strict local minima of $\Phi_{|A}$. Hence, there can be at most one
local minimum point.
\end{proof}

 We note that, given $c \in \cS$, there exists a
small ball $B $ in $\bbR ^W$ centered in $c$ such that $B\cap
[0,\infty)^W\subset \cS$. Hence,
  we think of $\cS$ as a manifold of dimension $|W|$ with boundary,
embedded in $\bbR^W$. Given $c \in \cS$, we write $\widetilde z (c)$
for the unique point of maximum of $\widehat{\mathcal J} (c, \cdot)$
in $\cA$ described in the above lemma. Then, the map $\widetilde z
(c)$ is regular:
\begin{Le}\label{pierbaubau}
The function $\widetilde z: \cS \rightarrow \cA$  is $C^1$, i.e.
there exists a $C^1$ function $g:\cU\rightarrow \cA$ from an open
subset $\cU\subset \bbR^W$ containing $\cS$, whose restriction to
$\cS$ coincides with $\widetilde z$.
\end{Le}

\begin{proof}
Let us first consider the smooth function $ G: \bbR^W  \times \cA
\rightarrow \bbR^\G$ defined as
$$ G( c, z)_\s =
\sum _{\s'\not = \s} \bigl( c(\s',\s)/z_{\s'}-c(\s,\s')
z_{\s'}/z_\s^2\bigr)\,.
$$
Let us fix $c \in \cS$ and write $\widetilde z $ for $\widetilde z
(c)$. Due to the computations in the proof of Lemma \ref{pierbau},
$G(c,z)= \nabla _z \Phi (z)$ for each $z \in (0,\infty)^\G$,  the
function $\Phi(z)$ being defined in (\ref{lorella08}). In
particular, $G(c, \widetilde z)=0$. Moreover, we know that the
tangent map
$T_z G(c, \cdot)$ from $T_z\cA$ to $\bbR^\G$  is a linear
monomorphism for $z=\widetilde z$, since the Hessian of
$\Phi_{|\cA}$ in $\widetilde z$ is strictly positive. We call
$V_{\widetilde z}$ the image of $T_{\widetilde z } \cA$ by the
tangent map $T_{\widetilde z } G(c,\cdot)$. Then $V_{\widetilde z}$
has dimension $\k-1=|\G|-1$ as  $T_{\widetilde z}\cA$. In
particular,
 there exists $ \widetilde \s \in \G$ such that $\p_{\widetilde \s}( W_{\widetilde z})$ has
 dimension $\k-1$, where $\p_{\widetilde \s}$ is the canonical orthogonal projection
$ \p_{\widetilde \s} : \bbR ^\G \rightarrow \bbR^{\G\setminus
\{\widetilde \s\}}$.
 This implies that the smooth composed map $H: \bbR^W \times
\cA \rightarrow \bbR^{\G\setminus \{ \widetilde \s\} }$ defined as
$H= \p_{\widetilde \s} \circ G$ has isomorphic tangent map
$T_{\widetilde z} H( c,\cdot) $ from $T_{\widetilde z} \cA$ to
$\bbR^{\G\setminus \{\widetilde \s\}}$.  Hence we can apply  the
Implicit Function Theorem for differential manifolds and conclude
the following: there exist a neighborhood $U$ of $c$ in $\bbR ^W$, a
neighborhood $V$ of $ \widetilde z$ in $\cA$ and a $C^1$ map
$f:U\rightarrow V$ such that i) $ H(c', f(c'))=0$ for all $c' \in U$
and ii) if $H(c',z)=0$ for some $c'\in U$ and $z \in V$ then
$z=f(c')$. At this point, we only need to  prove that
$G(c',f(c'))=0$ for all $c'\in U$ in order to conclude that
$\widetilde z (c')= f(c')$. If $H(c',z)=0$ then $G(c',z)_\s=0$ for
all $\s \in \G \setminus \{\widetilde \s\} $. On the other hand,
for all $(c', z) \in \bbR^W\times (0,\infty)^\G$ it holds $ \sum
_{\s } z_\s G(c', z)_\s =0$. Hence, if  $H(c',z)=0$, it must be
$$ G(c',z)_{\widetilde \s}= z_{\widetilde \s}^{-1} \sum _{\s:\s\not =\widetilde \s}
z_\s G(c',z)_\s=0\,,$$ thus concluding the proof.
\end{proof}

We give now another technical result which will be useful for the
proof in Appendix \ref{molofacciosec}:

\begin{Le} \label{vivairem}
Fix $c\in [0,\infty)^W$. Then,  $\cJ (c)=0$ if and only if
\begin{equation}
\sum_{\s'\in \G}c(\s,\s')=\sum_{\s'\in \G}c(\s',\s)\,.
\label{accolone}
\end{equation}
\end{Le}
\begin{proof}
Let us first prove the claim for $c \in \cS$. If \eqref{accolone} is
verified, then it is trivial to check that $z\in (0,\infty)^\G$ such
that $z_\s\equiv 1/\sqrt{|\G|}$ satisfies   \eqref{frignotta}, and
therefore it is its unique normalized solution. Then, by direct
computation, $\cJ(c)=\widehat \cJ( c,z)=0$.  If \eqref{accolone} is
not verified, then  the above $z$
 is not a solution of \eqref{frignotta}. This implies that
$\cJ(c)>\widehat \cJ(c,z)$. But trivially $\widehat \cJ(c,z)=0$,
thus implying that $\cJ(c)>0$.

We now  extend the result to general $c \in [0,\infty)^W$. To this
aim define $c_*\in [0,\infty)^W$ setting $c_*(\s,\s')\equiv 1$.
Trivially, $c_*$ belongs  to $\cS$ and satisfies \eqref{accolone}.
Suppose first that $c$ satisfies \eqref{accolone}. Then each vector
$\l c +(1-\l) c_*$ with $\l \in [0,1)$ satisfies \eqref{accolone},
and moreover belongs to $\cS$ having only positive entries. By the
first part, we conclude that $\cJ( \l c+ (1-\l) c_*)=0$ for all $\l
\in [0,1)$. Due to the continuity of $\cJ$ (see Lemma \ref{pierbau})
we conclude that $\cJ(c)=0$. Let us now suppose that $c$ does not
fulfill  \eqref{accolone} and prove that $\cJ(c)>0$.
 By convexity of $\cJ$ (see Lemma
\ref{pierbau}), for each $\l \in [0,1]$ it holds
\begin{equation}\label{saldi}
\cJ( \l c + (1-\l) c_* ) \leq \l \cJ(c) + (1-\l) \cJ(c_*)= \l \cJ
(c) \,. \end{equation} On the other hand, all vectors $\l c +(1-\l)
c_*$ with $\l \in [0,1)$ belong to $\cS$ and do not satisfy
\eqref{accolone}. Due to the first part, taking $\l=1/2$ we conclude
that $\cJ(1/2(c+c^*))>0$. This together with \eqref{saldi} implies
that $\cJ(c)>0$.
\end{proof}

\bigskip

Let us now come back to the variational problem (\ref{pranzo}) and derive some results about it from the
previous observations.

\smallskip

 Given a nonnegative measure $\p$ on $\G$ and nonnegative numbers  $ r(\s,\s')$, $(\s,\s')\in
 W$,
  we define  $c[\p,r] \in [0,\infty)^W$ as
 \begin{equation}\label{laviniagiulia}
  c[\p, r ] (\s,\s'):= \p(\s) r (\s,\s')\,, \qquad \forall
 (\s,\s')\in W\,.
 \end{equation}
Then (recall (\ref{pranzo}) and \eqref{emma28})
$$ j(\p, r)=\cJ\bigl(c[\p,r]\bigr)
$$
and the previous results on the variational problem associated to
$\cJ$ give information on the variational problem associated to $j$.
In particular, we stress that due to (\ref{laviniagiulia}) the
system of identities (\ref{frignotta}) coincides with the
stationarity of the measure $\p$ w.r.t. the  Markov generator
$R(\cdot , \cdot | z)$ on $\G$  defined as
$$ R (\s,  \s'|z)=
\begin{cases}
r(\s, \s') \frac{ z_{ \s' } }{ z_\s } \,, &
\text{ if }  \s \not =  \s'\,,\\
- \sum _{ \widetilde \s\,:\,\widetilde \s\not = \s } r(\s,\widetilde
\s) \frac{ z_{\widetilde \s}}{z_\s} \,,  & \text{ if } \s=  \s'\,.
\end{cases}
$$
 Recall the definition of the function $\widetilde z: \cS
\rightarrow \cA$. Given a nonnegative  measure $\p$  on $\G$ and $(x,s) \in
\bbR^d\times [0,T]$, we define
\begin{equation}\label{radiodue}
\widetilde z (\p, x,s) := \widetilde z \bigl( c\bigl[\p,r(\cdot,\cdot|x,s)
\bigr] \bigr)
\end{equation}
if $c\bigl[\p,r(\cdot,\cdot|x,s) \bigr] $ belongs to $\cS$.

\smallskip

Let us now take a strictly positive measure  $\p$ on $\G$, i.e.
$\p(\s)>0$ for all $\s \in \G$. Then, by Assumption (A2)
$c\bigl[\p,r(\cdot,\cdot|x,s) \bigr] $ belongs to $\cS$. If the
quasistationary measure $\mu(\cdot|x,s)$ is reversible for the
chemical generator $L_c(x,s)$ with jump rates $r(\s,\s'|x,s)$, then
the solution $\widetilde z\in\cA$ of the system (\ref{frignotta})
with $c= c\bigl[\p,r(\cdot,\cdot|x,s) \bigr] $ can be computed
explicitly. Indeed, setting
\begin{equation}
\widetilde z _\s =Z \sqrt{  \frac{\p(\s)}{\mu(\s|x,s)} } \,,\qquad
\s \in \G
\end{equation}
(where $Z$ is the normalizing constant assuring that $\widetilde z \in
\cA$), one has the detailed balance equation
\begin{equation}\label{signorina}
\p(\s) r (\s,\s'|x,s) \frac{\widetilde z_{\s'}}{\widetilde z_\s }=\p(\s')
r(\s',\s|x,s) \frac{\widetilde z_\s}{\widetilde z_{\s'}}\,,
\end{equation}
and consequently the validity of \eqref{frignotta}. Due to Lemma
\ref{pierbau} we conclude that
\begin{multline}\label{cammina}
j(\p , r(\cdot,\cdot|x,s)) = \widehat\cJ
\bigl(c\bigl[\p,r(\cdot,\cdot|x,s)
\bigr], \widetilde z \bigr)=\\
\sum _{(\s,\s')\in W}   \p (\s) r(\s,\s'|x,s) \left(1- \sqrt{ \
\p(\s') \mu(\s|x,s) \over \p(\s) \mu(\s'|x,s) }\right)\,.
\end{multline}
If $\p$ is not strictly positive, we can take a sequence  $\p_n$ of
strictly positive measure on $\G$ such that $\p_n(\s) \rightarrow
\p(\s)$ for each $\s \in \G$. This implies that
$$ c\bigl[ \p_n , r(\cdot,\cdot|x,s) \bigr] \rightarrow c\bigl[ \p, r (\cdot,\cdot
|x,s)\bigr] \,.$$ Since, as proved in Lemma \ref{pierbau}, $\cJ$ is
continuous on $[0,\infty)^W$, we obtain that
$$
j(\p_n , r(\cdot,\cdot|x,s)) = \cJ
\Bigl(c\bigl[\p_n,r(\cdot,\cdot|x,s) \bigr] \Bigr)\rightarrow   \cJ
\Bigl(c\bigl[\p ,r(\cdot,\cdot|x,s) \bigr] \Bigr)=j(\p ,
r(\cdot,\cdot|x,s))\,.$$ The above limit allows to extend
(\ref{cammina}) also to the case of general $\p\in [0,\infty)^\G$
with the convention to set $\p (\s) r(\s,\s'|x,s) \Bigl(1- \sqrt{ \
\p(\s') \mu(\s|x,s) \over \p(\s) \mu(\s'|x,s) }\Bigr)$ equal to zero
if $\p(\s)=0$.

This facts imply at once (\ref{variazione2}) assuming
\eqref{variazione1}

\section{Proof of Theorem \ref{LLN} }\label{dim_LLN}

Our proof of the law of large numbers in the time interval $[0,T]$
is based on a two scales argument. We give some  comments on  our
strategy for what concerns  the mechanical  evolution, similar
arguments hold for the chemical one.  We first divide the interval
$[0,T]$ in $M$ subintervals $I_k= [k\d, (k+1) \d]$, $\d:=T/M$, $k
\in \{0,1, \dots, M-1\}$. We  denote by $P^\l _{x,\s,t}$ the law of
the PDMP starting at $(x,\s)$ at time $t$ with $\l$--accelerated
chemical jumps, and by $x_*(\cdot|x,t)$ the solution  of the Cauchy
system
\begin{equation}\label{triestebis}
\begin{cases}
\dot z (s) = \bar F (z(s),s) \,, \qquad s \geq t \,, \\
z(t) =x\,.
\end{cases}
\end{equation}
We recall that $P^\l _{x_0, \s_0, 0}= P^\l _{x_0, \s_0}$. Then we
prove that  given  $\b\in (0,1)$ there exists $M=M(\b)$ such that
for each $k \in \{0,1, \dots , M-1\}$ the following holds: starting
at time $k\d$ in an arbitrary state $(x',\s')$ the random mechanical
trajectory  $\bigl(x(t)\,:\,t \in I_k\bigr)$
 deviates from
$\bigl( x_*(t|x' ,k \d)\,:\, t\in I_k)$ typically less than $\beta
\delta$. This is the content of Lemma \ref{funghetto} below. Having
this result we can derive (\ref{lim1}) of Theorem \ref{LLN} as
follows: in order to compare the mechanical trajectory
$\bigl(x(t)\,:\, t\in [0,T]\bigr)$  with $\bigl(x_*(t|x_0,0)\,:\,
t\in [0,T]\bigr)$ ($x(t)$ being now the random mechanical trajectory
when starting at state $(x_0,\s_0)$ at time zero) we fix $\b>0$,
take $M=M(\b)$ as above and by means of Lemma \ref{funghetto} for
each $k\in \{0,1,\dots, M-1\}$ we compare $x(\cdot)$ restricted to
the time interval $I_k$ with the path $\bigl(x_*(t|x(k\d), k\d)\,:\,
t\in I_k\bigr)$. As second step, we compare this last path with
 $\bigl(
x_*(t|x_0, 0)\;:\, t\in I_k\bigr)$. Due to the Lipschitz property of
the force fields, we will  show  below that
\begin{multline}\label{capitano}
\sup _{t\in I_k} |x_*(t|x(k\d),
k\d)-x_*(t|x_0,0)|\leq\\
C |x_*(k\d|x(k\d), k\d)-x_*(k\d|x_0,0)|=C | x(k\d)
-x_*(k\d|x_0,0)|
\,. \end{multline} This bound allows to implement by a recursive
procedure all the above estimates going from one $\d$--subinterval
to the next one.

In addition to $\delta$, another scale plays a crucial role. Indeed,
in order to prove  Lemma \ref{funghetto} we first divide each time
interval $I_k=[k\d,(k+1)\d]$ in $N$ subintervals $\{I_{k,n}\}_{0\leq
n<N}$,  where $I_{k,n}= [k\d+n\e,k\d+(n+1)\e]$ and $\e:=\d/N$. Then
we prove that the PDMP on $I_{k,n}$  obtained from the original one
by freezing the chemical jump rates at time $k\d+ n \e$ has a not
too large entropy w.r.t. to the original PDPM on $I_{k,n}$. This
entropy estimate allows to bound the probability for a deviation of
order at least $\b \e$ of the mechanical trajectory on $I_{k,n}$
from the expected asymptotic one.

\bigskip

Let us now enter into the technical details of the proof:

\bigskip

\begin{Le}\label{funghetto}
Fix  a  constant  $\b\in (0,1)$, a continuous function
$f:[0,T]\rightarrow \bbR$ and a compact set $\cK \subset \bbR^d$.
Then there exists a positive integer $M$ such that
for all  $ \s_0,\s \in \G$, for all $k \in \{0,1, \dots, M-1\}$,
setting $\d=T/M$ it holds
\begin{align}\label{pimpa1}
& \lim _{\l \uparrow \infty}\sup _{x_0 \in \cK} P^\l_{x_0,\s_0,k\d}
\left( \sup _{k\d \leq t \leq k\d + \d} \bigl|
x(t)-x_*(t|x_0,k\d)\bigr|\geq \b \d
\right) =0 \,,\\
&  \lim_{\l \uparrow \infty}\sup _{x_0 \in \cK}
P^\l_{x_0,\s_0,k\d}\left( \left|\int_{k\d}^{k\d+\d}  f(t)
\,\left[\,\chi(\s(t)=\s) - \mu \left(\s\,|\,
x_*(t|x_0,k\d)\,,\,t\right) \,\right]\,dt \right|>\b \d
\right)=0\,.\label{pimpa2}
\end{align}
\end{Le}

\smallskip

 Before proving the above lemma, let us explain how to derive from it
Theorem \ref{LLN}:

\smallskip

\noindent {\sl Proof of Theorem \ref{LLN}}.


Let us prove (\ref{lim1}). We start with some general consideration.
  As proved in Lemma \ref{phelps}, there exists a compact $\cK'$ such that
$x(t) \in \cK'$  for all $t \in [0,T]$, $P^\l_{x_0,\s_0}$--a.s. and  for all $\l>0$. By the same lemma,
there exists a compact $\cK$ containing $\cK'$ such that $x_*(t|x',s)\in \cK$ for all
 $x' \in \cK'$ and $s<t $ in $[0,T]$.
  We take the positive constant $K$ as in \eqref{lipF} for $\cK$ as described above,
   and by taking $K$ large enough we assume that \eqref{lipF} is satisfied also by the averaged field $\bar F$.
Then, we take $M$ as in Lemma \ref{funghetto} (since we
are now only interested in proving (\ref{lim1}), we can fix the function $f$ in Lemma \ref{funghetto}
arbitrarily). We divide the interval $[0,T]$ in $M$ subintervals $I_k=[k\d,
(k+1)\d]$,  where $k=0,1,\dots, M-1$ and $\d=T/M$.
For each $k$, we define
\begin{align*}
& \Delta_k:=\sup _{t \in I_k }\bigl|x(t)-x_*(t|x_0,0)
\bigr|\,,\\
& \cA_k:=  \{ |x(t)- x_*(t|x(k\d), k\d )|<\b \d \, \; \forall t \in
I_k  \}\,.
\end{align*}
Due to Lemma \ref{funghetto} and the Markov property of PDMPs,
 we have that
\begin{equation}
 P^\l _{x_0,\s_0}(\cA^c_k)= E^\l _{x_0,\s_0}
 \bigl[ P^\l _{x_0,\s_0} \bigl( \cA_k^c | x(k\d), \s(k\d)\bigr )\bigr]
\leq \sup_{x \in \cK, \s \in \G } P^\l _{x,\s, k \d} (\cA_k^c) \rightarrow 0 \,.
\end{equation}
This implies that
 \begin{equation}\label{ring}
 \lim _{\l \uparrow \infty} P^\l _{x_0,
\s_0}(\cA^c)=0\,, \qquad \cA:=\cap _{k=0}^{M-1}
\cA_k\,.
\end{equation}

 Assuming the event $\cap_{j=0}^k \cA_j$ to be verified,
since $x_*(t|x_0,0)= x_*(t| x_*(k\d|x_0,0) , k\d)$ and applying Gronwall
inequality as in Lemma \ref{phelps}, we obtain that
\begin{multline}
\D_k
\leq \sup _{t \in I_k}\bigl|x(t)-x_*(t|x(k\d),k\d) \bigr|+\sup _{t
\in I_k}\bigl|x_*(t|x(k\d),k\d)-x_*(t|x_0,0) \bigr|
\\< \b\d+ e^{K\d} | x(k\d)- x_*(k\d|x_0,0)|\leq \b\d +e^{K\d}
\D_{k-1}\,.
\end{multline}
Due to the event  $\cap_{j=0}^k \cA_j$, we can iterate the above
procedure and conclude that $\D_{k-1} <\b\d +e^{K\d}
\D_{k-2}$ and so on. At the end we obtain that the event $\cap_{j=0}^k \cA_j$ implies the event
$$ \cB_k:=\bigl\{ \D_0 <\b\d\text{ and } \D_j< \b\d
+e^{K\d}\D_{j-1} \; \forall j=1,2,\dots , k\bigr\} \,.
$$
Setting $z= e^{K\d}$, it is simple to check by induction  that the
event $\cB_k$ implies that  for each $j=0,1,\dots ,k$ it holds
\begin{equation}\label{stralunato}
 \D_j \leq \b \d \bigl( 1+z+z^2+\cdots + z^j \bigr)\leq \b \d \, \frac{z^M-1}{z-1}=\b
\d\,\frac{ e^{KT}-1}{e^{K\d}-1} \leq
 \b (e^{KT}-1)/K
 \end{equation}
(in the last inequality we have used that $e^x-1 \geq x$ for any $x
\geq 0$).
Hence,  $\cA$ implies \eqref{stralunato} for all $j=0,1, \dots, M-1$  and therefore it implies that
\begin{equation}\label{ronfare}
\sup_{0\leq t \leq T} \bigl | x(t)- x_*(t|x_0,0)\bigr|\leq \b \bigl(
e^{KT}-1 \bigr)/K\,.
\end{equation}
Due to arbitrariness of $\b$, this implies (\ref{lim1}).

\bigskip

Let us now prove (\ref{lim2}). We take $M$ as in Lemma \ref{funghetto}, where $f$ is the same function appearing
 in Theorem \ref{LLN} and $\cK$ is defined  as above.
Using the same arguments as above, it is simple
to derive from  Lemma
\ref{funghetto} that $P^\l_{x_0,\s_0}
(\cC)= 1-o(1)$, where $\cC$ denotes the event \begin{multline*}
 \cC =
\Bigl\{ \Bigl|\int_{k\d}^{k\d+\d}  f(t) \,\Bigl[\,\chi(\s(t)=\s) -
\mu \bigl(\s| x_*(t|x(k\d) ,k\d),t\bigr) \,\Bigr]\,dt
\Bigr|\leq \b \d \,, \\
\forall k \in \{0,1,\dots, M-1\} \, \Bigr\} \,.\end{multline*}
By assumption (A3)  we know that there exists a constant $\k >0$ such that
\eqref{lipmu} holds for all $x,y \in \cK$ and for all $t \in [0,T]$. Hence,  we can estimate
\begin{multline}\label{calzini}
\left| \int_{k\d}^{k\d+\d}  f(t) \,\Bigl[ \mu \bigl(\s\,|\,
x_*(t|x(k\d) ,k\d)\,,\,t\bigr) -\mu \bigl(\s\,|\, x_*(t|x_0
,0)\,,\,t\bigr)\,\Bigr]\,dt\right| \leq\\
 \k \|f\|_\infty \d\, \sup_{t
\in I_k}\bigl| x_*(t|x(k\d) ,k\d) -  x_*(t|x_0 ,0)\bigr|\,.
\end{multline}
Trivially,
$$
\bigl| x_*(t|x(k\d) ,k\d) -  x_*(t|x_0 ,0)\bigr|\leq  \bigl|
x_*(t|x(k\d) ,k\d) - x(t)\bigr|+ \bigl|x(t) - x_*(t|x_0
,0)\bigr|\,.$$ If  $\cA$ is verified, the first addendum in the
r.h.s. is bounded by $\b \d$ (recall that $\cA$ implies $\cA_k$),
while the second addendum is bounded by $\D_k\leq \b (e^{KT}-1)/K$.
Therefore, we can conclude that whenever the event $\cA$ is verified
the r.h.s. of \eqref{calzini} is bounded from above by
$$
\k \|f\|_\infty \d\, ( \b \d+ \b (e^{KT}-1)/K)\,.
$$
Using the triangular inequality, we conclude that the event $\cA\cap
\cC$  implies that
\begin{multline}
\Bigl|\int_0^T f(t) \,\Bigl[\,\chi(\s(t)=\s) - \mu \Bigl(\s\,|\,
x_*(t|x_0 ,0)\,,\,t\Bigr) \,\Bigr]\,dt \Bigr|\leq\\
 M\b\d+M \k
\|f\|_\infty \d\, ( \b \d+ \b (e^{KT}-1)/K)= \\T\b \bigl[1+ \k
\|f\|_\infty
 (  \d+  (e^{KT}-1)/K)\bigr]= C(T,K,\k,f) \b\,.
\end{multline}
Due to the arbitrariness of $\beta$ and since $P^\l_{x_0, \s_0}(\cA
\cap \cC)=1-o(1)$, the above estimate implies (\ref{lim2}). This concludes the proof of the averaging principle stated in Theorem
\ref{LLN}.
 \qed

\bigskip

We can now concentrate on the core of the law of large numbers, given
by Lemma \ref{funghetto}:
\bigskip

\noindent {\sl Proof of Lemma \ref{funghetto}}. We stress that $\b $
has to be considered as a fixed constant. We will play with two
length scales: $\d$ and $\e$, defined below. $C,c',\widetilde c,
c_i, \dots$ will denote non random positive constants independent
from $\d$ and $\e$,  that can change from line to line and that can
depend on $\b$. For simplicity of notation we take $k=0$ (the
arguments remain valid in the general case).
As the reader can check, in order to prove Lemma \ref{funghetto} we
will consider the process only when the  mechanical trajectory
$x(t)$ lies  inside a given compact set (uniformly in the starting
point $x_0 \in \cK$).  Hence, at cost to take a larger Lipschitz
constant $K$ in \eqref{lipF},  we can assume
 \eqref{lipF}, (\ref{lipr}),\eqref{lipmu} and \eqref{lipgamma} to hold for all $x,y
\in \bbR^d$ with $\k=K$. This allows to much simplify the notation.

\bigskip

We first prove (\ref{pimpa1}) and explain how to choose $M$.  To
this aim we   observe that due to Lemma \ref{phelps}  there exists a
compact subset $\widehat \cK\subset \bbR^d$ such that $
x_*(t|x_0,0)\in \widehat \cK$ for each $t \in [0,T]$ and $x_0\in
\cK$.  We define the compact set  $\cK'$ as
\begin{equation}
 \cK':=\cup _{x \in \widehat \cK} B(x,T) \,,
\end{equation} where $B(x,T)$ denotes the closed ball centered at
$x$ with radius $T$.
Given a  pair $(x,t) \in \cK' \times [0,T]$ we consider the
continuous--time homogeneous Markov chain on $\G$ with transition
rates $ r(\s,\s'):= r(\s,\s'|x,t)$, $\s\not=\s'$.  Note that this
Markov chain is ergodic and has  $\mu(\cdot|x,t)$ as stationary
probability. We call $Q_{\widetilde \s,x,t}$ its law when starting
in the state $\widetilde \s$. We will use the following uniform
large deviation estimate
\begin{multline}\label{fabri}
\limsup  _{u\rightarrow \infty} \sup_{(x,t) \in \cK'\times [0,T],\;
\s ,\s' \in \G} \\\frac{1}{u} \ln Q_{\s',x,t } \left[ \left| \frac{1}{u
} \int_0 ^u \bigl[ \chi ( \s( s ) =\s) - \mu (\s|x,t) \bigr]ds
 \right|\geq  \frac{\b}{20 (a_* +1)|\G|} \right]=:  - W_1<0\,,
\end{multline}
where
\begin{equation}\label{aia}
a_*:= \max _{\s' \in \G} \max _{s \in [0,T]} \max _{x'\in \cK'}  |
F_{\s'} ( x',s) |\,,
\end{equation}
and $W_1$ is a strictly positive constant. The result \eqref{fabri}
follows from lemma \ref{molofaccio} in Appendix \ref{molofacciosec}.

We now introduce a second positive constant $W_2$ defined as
\begin{equation}\label{codroipo000}
W_2=K+2+a_*\,.
\end{equation}
Finally we define the constant $\d:=T/M$, where $M$ is the smallest
positive integer such that
\begin{equation}\label{merlino}
\d\leq  \min\{W_1/(4c_3), 2 \beta /(5W_2)\}\,,
\end{equation}
where the positive constant $c_3$ depends only on the compact $\cK$ and will
be defined in (\ref{nicotina}).
\bigskip

Let us now prove (\ref{pimpa1}) for $k=0$. In order to shorten the
notation we write $x_*(t)$ instead $x_*(t|x_0,0)$. We stress that
$x_*(t)$ depends on $x_0$, although $x_0$ has been omitted in the
notation. We define
the random time $\t$ in terms of  the exit time from the
$\b\d$--tube $\cA_{\b\d}$ around $x_*(t)$:
$$
 \t:=\inf \{ t\geq 0  \,:\, (x(t),t) \not \in \cA _{\b\d}
\}\,,
$$
where
\begin{equation}\label{finire} \cA_{\b \d} := \{ (x,t)\in \bbR
^d \times [0,\d]\,:\, |x-x_*(t)|\leq \b\d \}\,.
\end{equation}
Note that, since $\beta \delta \leq T$, for each $x_0\in \cK$ the
above tube $\cA_{\b \d }$ is included in $\hat \cK$.  We want to
prove that $\t\geq\d$ with probability $1-o(1)$ as $\l \uparrow
\infty$, which is equivalent to (\ref{pimpa1}).  Up to the Markov
time $\t$ the PDMP is determined only by its characteristics
restricted to $\cA_{\b\d}$. Since we will follow the process only up
to time $\t$, due to (\ref{lipF}) and  (\ref{lipr}), at cost of
changing the characteristics outside $\cA_{\b\d}$ without loss of
generality
 we can assume that
\begin{align}
& \left| F_\s (x,t)- F_\s( x_*(t),t) \right|\leq K \d \,,\label{vento1}\\
& \left| r(\s,\s'|x,t)-r(\s,\s'|x_*(t),t) \right|\leq K \d
\,, \label{vento2}\\
& \left |\g_\s (x,t)-\g _\s (x_*(t),t) \right|\leq K
\d\,,\label{vento3}
\end{align}
for all $x \in \bbR^d$ and $t \in [0,\d]$.
 Using (\ref{vento1}) and the fact that $\t\leq \d$, one
easily obtains that, given   $m=1,2,\dots d$,
$P^\l_{x_0,\s_0}$--a.s. it holds
\begin{multline}\label{acqua}
x(t)_m =x(0)_m  + \int _0 ^t F_{\s(s)} \bigl( x(s),s\bigr) _m
ds  \\
=x(0)_m + \int _0 ^t F_{\s(s)}\bigl( x_*(s),s \bigr)_m ds +\cE_1\,,
\qquad \forall t \leq \t\,,
\end{multline}
where the error term $\cE_1$ can be  bounded as $|\cE_1|\leq K\d^2$.

 Given an
integer $N$, we divide the interval $[0,\d]$ in $N$ subintervals of
length $\e:=\d/N$.
We now explain how to fix the constant $\e$.  The first requirement
is that
 $\e\leq \d^2$. Moreover, consider the functions $F_\s(\cdot, \cdot)$ on
$\cK'\times [0,T]$. Since they are uniformly continuous, there
exists $\e_1>0$ such that $ |F_\s(x_1,s_1)-F_\s(x_2,s_2)|\leq \d$
for any $x_1, x_2 \in \cK'$, $s_1,s_2 \in [0,T]$ such that
$|x_1-x_2|\leq \e_1$ and $|s_1-s_2|\leq \e_1$. We require that
$$ \e\leq \min \{ \e_1, \e_1/ C_0 \}\,, $$
where
$$C_0= \sup _{(x,s) \in \cK' \times [0,T] }|\bar F (x,s) |\,, \qquad
\bar F (x,s)= \sum _{\s\in \G} F_\s (x,s) \mu(\s|x,s) \,.
$$
 We note that this condition implies that
\begin{equation}\label{condimento1}
\bigl|F_\s(x_*(s),s)_m- F_\s (x_*(t),t)_m\bigr |\leq \d
\end{equation}
for all $m =1,\dots d$, all $\s \in \G$ and  all $ s,t, \in [0,T]$
such that $|s-t|\leq \e$. Indeed, for such $s,t$ it holds
$$ |x_*(s)-x_*(t)|= \bigl| \int_s^t \bar F ( x_*(u),u) du \bigr| \leq C_0 |s-t| \leq \e_1\,.$$
This allows to derive (\ref{condimento1}) from our choice of $\e_1$.
By  similar arguments we derive that for a suitable positive
constant $\e_2$ it holds
\begin{equation}\label{vladimir}
|r(\s,\s'|x_*(s),s)-r(\s,\s'|x_*(t),t)| \leq \d \end{equation} for
any   $s,t, \in [0,T]$ such that $|s-t|\leq \e_2$ and for any
$\s\not = \s'$ in $\G$.  We fix $\e$ s.t. $ \e \leq \e_2$.
Similarly, there exists a positive constant $\e_3$ such that
\begin{equation}\label{vissani}
|\mu(\s|x_*(s),s)-\mu(\s|x_*(t),t)| \leq \frac{\b}{20(a_*+1)|\G|}
\end{equation} for any $s,t, \in [0,T]$ such that $|s-t|\leq \e_3$
and for any $\s$ in $\G$.  We fix $\e$ s.t. $ \e \leq \e_3$.

 Writing
 $$
\int _0 ^t F_{\s(s)}\bigl( x_*(s),s \bigr)_m ds= \sum_{\s\in \G }
\int _0 ^t F_{\s}\bigl( x_*(s),s \bigr)_m \chi( \s (s) = \s) ds \,,
$$
due to  \eqref{acqua}, (\ref{condimento1}) and the condition   $\e
\leq \d^2$, it holds
\begin{equation}\label{acqua1}
x(t)_m  =x(0)_m + \sum _{\s \in \G} \sum _{j=0}^{\lfloor
t/\e\rfloor-1} F_\s (x_*(j\e),j\e)_m   \int _{j\e} ^{j\e+\e}
\chi(\s(s)=\s) ds +\cE_2 \,, \qquad  \forall t \leq \t\,,
\end{equation}
where the error term $\cE_2$ can be  bounded as $|\cE_2|\leq
|\cE_1|+\d^2+ \d^2 a_* $. Above, $ \lfloor t/\e \rfloor $ denotes
the integer part of $t/\e$ and  the sum over $j$ is set equal to
zero if $\lfloor t/\e \rfloor =0$. The condition $\e\leq \d^2$ is
necessary in order to bound the error term $ \int_{\lfloor
t/\e\rfloor \e} ^{t} F_{\s(s)}( x_*(s),s) _m ds$.

\smallskip

We claim that there exist  positive constants $c,c'$ independent of
$\e$ and $\d$ such that the $\l$--accelerated PDMP with unrescaled
characteristics satisfying (\ref{vento1}), (\ref{vento2}) and
(\ref{vento3}) fulfills the bound
\begin{equation}\label{stellina}
P^\l _{x',\s',j \e} \left( \left|\int _{j\e}^{j\e +\e}\bigl [\,\chi(
\s(s)=\s) - \mu(\s| x_*(s|x_0,0),s) \,\bigr]ds\right|>\frac{\b\e}{10
(a_*+1) |\G|}\right)\leq e^{-c \l \e}\,,
\end{equation}
for any $ j =1,2,\dots,N-1$, any  $x_0 \in \cK$, any  $ x' \in
\bbR^d$ s.t. $|x'-x_*(j\e|x_0,0)|\leq \b \d$, for any $ \s,\s' \in
\G$ and for any $\l \geq c'/\e$. We recall that $a_*$ has been
defined in (\ref{aia}). Above we  used again the notation
$x_*(\cdot|x_0,0)$ in order to stress the dependence on $x_0$.

\smallskip

 Before proving (\ref{stellina}) let us explain how it
allows to conclude the proof of Lemma \ref{funghetto}. First we show
that, due to (\ref{acqua1}),  (\ref{stellina}) and   the Markov
property,
 with probability at least $1-(\d/\e) e ^{-c\l \e}|\G|$ it holds
\begin{equation}\label{acqua2}
x(t)_m  =x(0)_m +
 \sum _{\s \in \G} \sum _{j=0}^{\lfloor t/\e \rfloor-1 }
 F_\s (x_*(j\e),j\e) _m\int _{j\e} ^{j\e+\e}
 \mu ( \s|x_*(s),s) ds +\cE_3\,, \qquad  \forall
t \leq \t \,,
\end{equation}
where the modulus of the  error $\cE_3$ can be bounded  as
$|\cE_3|\leq |\cE_2|+\b \d /10$.  To this aim, we define the event
$\cB_j(\s)$ for $j=0,1,\dots, N-1$ as $ \cB_j (\s) := \cC _j (\s)
\cap \{j\e+\e\leq \t\}$ where
\begin{equation}\label{elezioni}
 \cC_j (\s) :=\left
\{\left|\int _{j\e}^{j\e +\e}\bigl [\,\chi( \s(s)=\s) - \mu(\s|
x_*(s),s) \,\bigr]ds\right|>\frac{\b\e}{10 (a_*+1) |\G|}\right\}\,.
\end{equation}
Conditioning on time $j\e$ and using the Markov property we can
estimate
$$P^\l_{x_0,\s_0} ( \cB_j(\s) ) = P^\l_{x_0, \s_0} ( \cC_j(\s)
\,;\,j\e+\e\leq \t )= E^\l_{x_0, \s_0} \left[ P_{x(j\e), \s(j\e), j
\e} ( \cC_j (\s) )\chi(j\e +\e \leq \t)\right]\,.$$ At this point,
we observe that the condition $j\e\leq \t$ implies that
$|x(j\e)-x_*(j\e)|\leq \b \d$, thus allowing to estimate the
probability $ P_{x(j\e), \s(j\e),j\e} ( \cC_j (\s) )$ from above by
$ e^{-c\l \e}$ due to (\ref{stellina}). In particular, we obtain
that
$$ P^\l _{x_0, \s_0} ( \cB ) \leq N |\G| e^{-c\l \e} \,, \qquad \cB:
= \cup _{\s\in \G} \cup _{j=0}^{N-1} \cB _j (\s)\,.
$$
Finally, we note that the event $\cB^c$ implies  for any $\s$ that
$C_j(\s)$ is not fulfilled whenever the interval $[j\e, j\e +\e]$ is
included in $[0, \t]$. Hence, in this case in (\ref{acqua1}) one can
substitute $\chi ( \s(s) = \s )$ by  $\mu(\s|x_*(s),s)$   with an
error bounded by $\b \e N /10= \b \d/10  $. This leads to
(\ref{acqua2}).

\smallskip

By applying again (\ref{condimento1}), (\ref{acqua2}) implies that
\begin{equation}\label{acqua3}
x(t)_m  =x(0)_m + \sum_{\s\in \G}\int_0 ^t F_{\s} (x_*(s),s)_m \mu (
\s|x_*(s),s)ds +\cE _4= x_*(t)_m+  \cE _4\,, \qquad \forall t \leq
\t\,,
\end{equation}
where the error $\cE_4$ can be bounded as $|\cE_4|\leq
|\cE_3|+\d^2$. Collecting all the previous estimates, we conclude
that for all $x_0\in \cK$ it holds $ |\cE_4| \leq \b\d/10+ W _2
\d^2$, the constant $W_2$ being defined in (\ref{codroipo000}).
Hence, due to our choice (\ref{merlino}), we can conclude that
$|\cE_4| \leq \b \d /2$.
 Hence, taking $t=\t$ in (\ref{acqua3}), we conclude that with
probability $1-(\d/\e) e ^{-c \l \e}|\G|$ it must be $ |x(\t) -
x_*(\t)|\leq \b \d /2$. Due to the continuity of the mechanical
trajectories and the definition of $\t$, this implies that $\t=\d$.
Coming back to (\ref{acqua3}) with this additional information we
get (\ref{pimpa1}) with $k=0$. As already remarked, the same
arguments allow to prove (\ref{pimpa1}) for a generic $k $.

\smallskip We point out that by the above method we have  approximated the random
 integral
 $$
\int _{k\d} ^{k\d+\d} F_\s (x (s), s) _m \chi (\s(s)=\s) ds$$
 by the new integral
$$\int _{k\d} ^{k\d+\d} F_\s (x_*(s),s)_m \mu(\s|x_*(s),s)ds\,.
$$ Hence, the proof of (\ref{pimpa2}) is completely analogous, since
it is enough to replace the force field with the test function
$f(s)$.


\bigskip

It remains now to  prove  (\ref{stellina}). In order to simplify the
notation, we write $P^\l$ for the law $P^\l _{x',\s',j \e}$ of the
$\l$--accelerated PDMP having unrescaled characteristics that
 satisfy (\ref{vento1}), (\ref{vento2}) and (\ref{vento3}), starting
in the state $(x',\s')$ at time $j\e$ and evolving up to time
$j\e+\e$. In addition, we write  $Q^\l$ for the law of the
$\l$--accelerated  PDMP restricted to the time interval $[j\e,
j\e+\e]$, starting in the state $(x',\s')$ at time $j\e$ and with
characteristics $\bigl(F ,\l \widetilde r\bigr)$, where the new
unrescaled transition rates $\widetilde r$ are constant and are
defined as
\begin{align*}
& \widetilde r (\s_1, \s_2):=\widetilde p(\s_1,\s_2) \widetilde \g_{\s_1}\,,\\
& \widetilde p(\s_1,\s_2):= p(\s_1,\s_2\,|\, x_*(j\e) , j \e)\,,\\
&  \widetilde \g_{\s} := \g_{\s}( x_*(j\e), j \e)\,.
\end{align*}
Namely, the above rates correspond to the original rates read along
the asymptotic trajectory $x_*$ and  frozen at time $j\e$.
Note that this new PDMP is not coupled: the chemical evolution is a
continuous--time homegeneous  Markov chain, while the mechanical
evolution is a function of the chemical one.  One can compute and
bound the Radon--Nikodym derivative $dP^\l/dQ^\l$. Indeed, if
$\bigl(x(t),\s(t)\bigr)$ is an element in the Skohorod space
$D\bigl( [j\e, j\e+\e], \bbR^d \times \G\bigr)$ with $n$ jumps at
times $\t_1<\t_2<\cdots < \t_n$, setting $\t_0=j \e $, $\s_0:=\s'$,
$x_0 := x'$ and
$$ \s _i :=\s(\t_i)\,, \; x_i := x( \t_i)\,, \; i=1,2, \dots,
n\,,$$ one has
\begin{multline}\label{plasmon}
\frac{dP^\l}{dQ^\l} \left(x(\cdot),\s (\cdot)\right)=\\
\left[ \prod _{i=1} ^n \frac{ r(\s_{i-1}, \s_i |x(\t_i),\t_i )
}{\widetilde r ( \s _{i-1} ,\s_i) }\right]
 \exp
 \left\{ - \l \int_{j\e} ^{j\e+\e}
\left( \,\g_{\s(s) } ( x(s),s    )- \widetilde \g _{\s(s)} \right)
ds\right\}\,.
\end{multline}
The term in the square bracket can be rewritten as
\begin{equation}\label{plasmonbis}
\prod _{i=1} ^n \frac{ r (\s_{i-1} , \s_i |x(\t_i),\t_i )
}{\widetilde r (\s _{i-1} ,\s_i) }= \exp \left\{\sum_{i=1}^n
 \left[    \ln  r (\s_{i-1} ,\s_i |x(\t_i),\t_i )  - \ln \widetilde r (\s _{i-1} ,\s_i)\right]  \right\}\,.
\end{equation}
Due to \eqref{lipr}, (\ref{vladimir})  and the assumption $\e\leq
\e_2$ we conclude that
\begin{eqnarray*}
& &  \left|\, r (\s_{i-1} ,\s_i |x(\t_i),\t_i )  - \widetilde r (\s
_{i-1} ,\s_i)\,\right| = \left|\, r (\s_{i-1} ,\s_i |x(\t_i),\t_i )
-  r (\s _{i-1} ,\s_i| x_*(j\e),j\e)\,\right| \\
 & & \leq\left|\, r (\s_{i-1} ,\s_i |x(\t_i),\t_i )
-  r (\s _{i-1} ,\s_i| x_*(\t_i),\t_i)\,\right|\\
& & +\left|\, r (\s_{i-1} ,\s_i |x_*(\t_i),\t_i ) -  r (\s _{i-1}
,\s_i| x_*(j\e),j\e)\,\right|\leq (K+1)\d\,.
\end{eqnarray*}

By applying Taylor expansion to the $\log$ function and using
(\ref{plasmonbis}), we get that the term in the square
 bracket in (\ref{plasmon}) is bounded from above by $e^{c_0 n  \d}$
 where $n$ denotes the number of chemical jumps in the interval
 $[j\e,j\e+\e]$. The constant $c_0$ does not depend on $\b, \e, \d,j$ and
is the same for all $x_0\in \cK$ and all pairs $(x',\s')$ as in
\eqref{stellina}.
 Similarly one gets that the exponential in
(\ref{plasmon}) is bounded from above by $e^{ c_1\l \e\d} $, where
the constant $c_1$   does not depend on $\b, \e, \d,j$ and is the
same for all $x_0\in \cK$ and all pairs $(x',\s')$ as in
\eqref{stellina}.
 Moreover, setting
$$ C:= \max _{\s \in \G} \max_{(x,s) \in \cK' \times [0,T]} \g_\s (x,s)\,,$$
 the random variable
 $n$ is stochastically dominated by a Poisson random variable  $Z$ with
 mean  $C\l \e$. Recalling that
 $E(e^{aZ})=e^{ C\l \e (e^a-1)}$  and taking $a=2c_0 \d$,  we conclude that
 \begin{equation}\label{nicotina}
 E_{Q^\l} \left( \left| \frac{d P^\l}{d Q^\l}\right|^2 \right) \leq
 e^{2 c_1\l\e \d+ C \l \e ( e^{2c_0\d}-1)}\leq e^{c_3 \l \e \d}\,, \quad c_3:=
2c_1 + 2c_0 C \frac{e^T-1}{T}\,.
\end{equation}
Above we have used the inequality $e^x-1\leq \frac{e^T-1}{T} x $
valid for all $x\in [0,T]$, which follows from the convexity of $x\rightarrow e^x$.

Due to (\ref{vissani}) and our assumption $\e \leq \e_3$ we can
bound
$$
\int _{j\e} ^{j\e+\e} \left| \mu (\s|x_*(s) ,s )- \mu (\s |x_*(j\e),
j\e) \right|ds \leq \frac{\b\e}{20(a_*+1)|\G|}\,.
$$
 Hence, calling $\cD$ the event
$$ \cD:=\left\{  \left|\int _{j\e}^{j\e +\e}\bigl
[\,\chi(\s(s)=\s)  - \mu(\s| x_*(j\e),j\e)
\,\bigr]ds\right|>\frac{\b\e}{20 (a_*+1) |\G|}\right\}
$$
in order to conclude the proof of (\ref{stellina}) we  need to
 bound $P^\l (\cD)$.
To this aim, we  write $Q_{\s'}$ for the law of the continuous--time
Markov chain on $\G$ starting at $\s'$, jumping from $\s_1$ to
$\s_2$  with transition rate $\widetilde r ( \s_1, \s_2)$.
 Note that
$\mu(\cdot |x_*(j\e), j\e)$ is the invariant measure for this Markov
chain. Then
\begin{equation}\label{fabrixxx}
Q^\l \left[\cD\right]=
 Q_{\s'}\left[ \left|  \frac{1}{\l\e} \int_0 ^{\l\e} \bigl[
\chi ( \s( s ) =\s) - \mu (\s|x_*(j\e), j\e) \bigr]ds
 \right|> \frac{\b}{20 (a_* +1)|\G|} \right]
\end{equation}
We note that $x'$ as in \eqref{stellina} must belong to the compact
set $\cK'$. Hence, due to (\ref{fabri}), if $\l \e\geq c_4$ ($c_4$
being independent on $\e,\d$ and being the same for all $(x',\s')
\in \cK' \times \G$), then
\begin{equation}\label{brunbrun}
 Q^\l \left[ \cD\right] \leq
e^{- \e \l W_1 /2}\,,
\end{equation}
where $W_1$ has been defined in (\ref{fabri}).

 Finally, we can apply Schwarz inequality together with (\ref{nicotina}) and (\ref{brunbrun})
  in order to
conclude that
\begin{multline}
P^\l [\,\cD\,]= E_{Q^\l}\Bigl[ \, \frac{dP^\l }{d Q^\l} \cdot\chi
_{\cD}\Bigr]\leq E_{Q^\l}\Bigl[ \,\Bigl( \frac{dP^\l }{d Q^\l}
\Bigr)^2\Bigr]^{1/2} Q^\l (\cD)^{1/2} \leq
 \exp\Bigl\{ -\l \e \bigl[W_1/4- c_3  \d/2  \bigr] \Bigr\}\,.
\end{multline}
Since by  definition (\ref{merlino})  $c_3 \d /2 \leq W_1/8$, we
obtain that $P^\l [\cD]\leq e^{-\l \e W_1/8}$. This implies
(\ref{stellina}) with $c=W_1/8$ and $c'=c_4$.

\qed

%

\section{Proof of  Theorem \ref{LDP}}\label{dim_LD}

We have now all the tools in order to prove the LDP.  We recall that
in Section \ref{var_problema} we proved (\ref{variazione2}) assuming
\eqref{variazione1}. Here we start by analyzing the Radon--Nikodym
derivative of the PDMP w.r.t. a perturbed version. To this aim let
$V=\bigl\{V_\s\bigr\}_{\s\in \G} $ be a family of $C^1$ functions
$V_\s:[0,T]\rightarrow \bbR$ parameterized by $\s \in \G$ and call
$V'_\s(s):=\frac{d V_\s (s)}{d s}$ the corresponding derivatives. We
introduce some perturbed rates according to the following
definitions
\begin{align} & \widetilde r (\s,\s'|x,s):= r
(\s, \s'|x,s) e ^{V_{\s'}(s)-V_\s(s)}\,,\label{alig}
\\
& \widetilde \g_\s(x,s):= \sum _{\s'\in \G} \widetilde
r(\s,\s'|x,s)\,. \nonumber
\end{align}
Writing $P^{\l,V}_{x_0,\s_0}$ for the law of the  PDMP with
characteristics $\bigl(F, \l \widetilde r\bigr)$ and denoting by
$\t_1< \t_2<\dots <\t_n$
 the jump times of the path $\s(\cdot)$ in the time interval $[0,T]$, it holds
\begin{multline}\label{RN}
\frac{dP^\l_{x_0,\s_0}}{d P^{\l,V}_{x_0,\s_0}}\bigl(x, \s\bigr)=\exp
\left\{ \sum _{i=1}^n [V_{\s(\t_i-)}(\t_i)-V_{\s(\t_i)}
(\t_i)]\right\}
\cdot\\
\exp\left\{ -\l \sum _{\s'\in \G} \int _0^T  r(\s(s) ,\s'|x(s),s)
 \bigl( 1- e
^{ V_{\s'}(s)-V_{\s(s)}(s)} \bigr) ds\right\}
 \,.
\end{multline}
In  order to estimate the first exponential in (\ref{RN}) we observe
that
$$
\sum _{i=0}^n \int_{\t_i}^{\t_{i+1}} V'_{\s(s)}(s) ds =
\sum _{i=0}^n \left[ V_{\s (\t_{i+1}-)}(\t_{i+1} ) -V_{\s(\t_i)}(
\t_i )\right]\,,
$$
where we set $\t_0:=0$ and $\t_{n+1}:=T$.
This implies that
$$
\exp \left\{ \sum _{i=1}^n [V_{\s(\t_i-)}(\t_i)-V_{\s(\t_i)}
(\t_i)]\right\}=\exp \left\{ -\left[ V_{\s(T-)} (T)-V_{\s(0) }(0)-
\int_0 ^T V'_{\s(s)}(s) ds \right]\right\}\,.
$$
Writing $Q^{\l,V}_{x_0,\s_0}$ for the law of the perturbed process
on $\Upsilon$ the above computations give
\begin{equation}\label{speriamoinbene}
\frac{ d Q^\l_{x_0,\s_0}}{d Q^{\l,V}_{x_0,\s_0}} \bigl(x,
\rho\bigr)= \exp \left\{ -\l \left[ J_V(x,\rho) +O(1/\l)
\right]\right\}\,,
\end{equation}
where
\begin{equation}
J_V(x,\rho)=  \sum_{\s \in \G}\sum_{\s'\in\G} \int_0^T \rho(s)_\s
 r(\s,\s'|x(s),s)
 \bigl( 1- e
^{ V_{\s'}(s)-V_\s(s)} \bigr)ds
\end{equation}
and $O(1/\l)$ denotes a quantity bounded in modulus by $c/\l$, $c$
being a   positive constant depending only on $\{V_\s\}_{\s\in \G}$ and $T$.

We point out that the function $J_V:\Upsilon \rightarrow \bbR$  is
continuous. Indeed, if $(x^{(n)}, \rho ^{(n)})$ converges to $(x,
\rho)$ then \begin{multline}
 \left|J_V\bigl(x^{(n)}
,\rho^{(n)}\bigr)- J_V(x, \rho) \right|\leq \\
\sum_\s \sum_{\s'} \left|\int_0 ^T \bigl( \rho_\s ^{(n)
}(s)-\rho_\s(s)\bigr) r(\s,\s'|x(s),s)
 \bigl( 1- e
^{ V_{\s'}(s)-V_\s(s)} \bigr)ds\right|+ \\c(V) \sum_\s \sum_{\s'}
 \int_0^T \rho^{(n)}_\s(s)
 \left| r(\s,\s'|x(s),s)- r(\s,\s'|x^{(n)}(s),s) \right|ds\,.
\end{multline}
Since $\rho^{(n)  }_\s \rightarrow \rho_\s$ in $L[0,T]$ and since
the transition rates are continuous, the first expression in the
r.h.s. goes to zero as $n\rightarrow \infty$. Since
$\|x^{(n)}-x\|_\infty\rightarrow 0$, by the continuity of the
transition rates and the Dominated Convergence Theorem the second
expression in the r.h.s. goes to zero, thus concluding the proof of
the continuity of $J_V$.

\smallskip

Consider the function $J(x,\rho)$ defined in (\ref{variazione1}).
This can be rewritten as
\begin{equation}\label{variazione1bis}
J  ( x, \rho )= \sup _z \sum_\s \sum _{\s' } \int_0 ^T \rho_\s (s) r
(\s,\s'|x(s),s )\left [ 1-\frac{z_{\s'}(s)}{z_\s(s)} \right] ds \,,
\end{equation}
where the supremum is taken over the family of  measurable functions
$\{z_\s\}_{\s\in \G}$,  $ z_\s:[0,T]\rightarrow (0,\infty)$.
Approximating measurable functions by bounded $C^1$ functions,
 we obtain that the functional
$J(x,\rho)$ defined in (\ref{variazione1}) can be expressed as
\begin{equation}\label{avent}
J(x,\rho)= \sup_V J_V(x,\rho)\,, \qquad (x,\rho) \in \Upsilon\,.
\end{equation}

\noindent $\bullet$ {\em Regularity of $J$}. Trivially
$J(x,\rho)<\infty$ for each $(x,\rho)\in \Upsilon$.
Moreover, $J$ is lower semi--continuous since due to the previous
observations it is the supremum of the family of continuous
functions $J_V$.

\bigskip

\noindent $\bullet$ {\em Proof of the upper bound
(\ref{spiegelupper})}. Let us start with a generic subset  $U\subset
\Upsilon$. Given a family $V$ of $C^1$ functions $\{V_\s(s) \}_{\s\in \G}$,
we can bound
\begin{multline*}
Q^{\lambda}_{x_0,\s_0}(U) =  Q^{\lambda,V }_{x_0,\s_0}\left(\frac{d
Q^{\lambda}_{x_0,\s_0}}{d Q^{\lambda,V}_{x_0,\s_0}} \chi_U\right)
\leq e^{-\lambda\inf_{(x,\rho)\in
U}\left[J_V(x,\rho)+O\left(\lambda^{-1}\right)\right]} Q^{\lambda,V}_{x_0,\s_0}(U) \leq \\
  e^{-\lambda\inf_{(x,\rho)\in
U}\left[J_V(x,\rho)+O\left(\lambda^{-1}\right)\right]}\,.
\end{multline*}
This implies that
$$
\limsup_{\lambda \to \infty}{1 \over \l}\log
Q^{\lambda}_{x_0,\s_0}(U)\leq -\inf_{(x,\rho)\in U}J_V(x,\rho)\,.
$$
This holds for any choice of the functions $V_\s$. Optimizing over
the $V_\s$ we get
\begin{equation}\label{fleurs}
\limsup_{\lambda \to \infty}{1\over \l}\log
Q^{\lambda}_{x_0,\s_0}(U)\leq -\sup_{V}\inf_{(x,\rho)\in
U}J_V(x,\rho)\,.
\end{equation}
At this point, we would like to invert the supremum and the infimum
in the r.h.s. if $U$ is given by some closed subset $C\subset
\Upsilon$.
  To this aim we observe that (i) $C$ is compact since it is a
closed subset of the compact space $\Upsilon$ (see Lemma
\ref{annozero}), (ii) estimate (\ref{fleurs}) holds for each
$U\subset \Upsilon$ and in particular for each open subset $U\subset
\Upsilon$, (iii) $J_V(x,\rho)$ is continuous on $\Upsilon$ for each
$V\in C^1[0,T]^\G$. Hence, we can apply Lemma 3.3 in \cite{KL} (note
that $J_\b(\mu)$ there coincides with our $-J_V(x,\rho)$), which
together with (\ref{avent}) implies the upper bound
$$
\limsup_{\lambda \to \infty}{1 \over \l}\log
Q^{\lambda}_{x_0,\s_0}(C)\leq -\inf_{(x,\rho)\in C}\sup _V
J_V(x,\rho)=-\inf_{(x,\rho)\in C}J(x,\rho)=-J(C)\,.
$$

\bigskip

\bigskip

\noindent $\bullet$ {\sl Proof of the lower bound
(\ref{spiegeldown}).}

We first introduce a special subset $\cB$ of $\Upsilon$ as $$\cB =
\bigl \{ (x,\rho)  \in \Upsilon \,:\, \rho _\s \in C^1[0,T] \text{
and }  \rho_\s(t)>0  \;\forall \s \in \G,\; t \in [0,T]\bigr\}\,.
$$
As shown in Lemma \ref{ucraina} in the Appendix, $\cB$ is a dense
subset of $\Upsilon$.

Let $O$ be an open subset of $\Upsilon$ and fix $(x^*,\rho^*)\in
O\cap \cB $ (note that $(x^*,\rho^*)$ exists since $\cB$ is dense in
$\Upsilon$).
 Define
\begin{equation}
\widetilde{V}_\s(s) = \ln \widetilde z _\s\bigl( \rho^*_\s (s), x^*(s),s)\,,
\end{equation}
where $\widetilde z _\s$ has been defined in (\ref{radiodue}). Since
rates are assumed to be $C^1$ (see assumption (A3)), since $\{
\rho^*_\s (s) r(\s,\s' |x^*(s), s)\}_{\s,\s'}$ belongs to the set
$\cS$ defined at the beginning of  Section \ref{var_problema} (see
assumption (A2))
 and due to
Lemma (\ref{pierbaubau}) , we get that $\widetilde{V}_\s \in
C^1[0,T]$. We take as perturbation
$\widetilde{V}=\{\widetilde{V}_\s\}_{\s \in \G}$.

Due to the fact that $O$ is open we have for any $\delta$ small enough
\begin{equation*}
Q^{\lambda}_{x_0,\s_0}(O)\geq
Q^{\lambda}_{x_0,\s_0}(B_{\delta}(x^*,\rho^*))\,, \qquad \forall \l
>0\,,
\end{equation*}
where $B_{\delta}(x^*,\rho^*)$ is the ball in $\Upsilon$ of radius
$\delta$ and center $(x^*,\rho^*)$. We now use  the following
estimate
\begin{multline}\label{alice}
Q^{\lambda}_{x_0,\s_0}\left(B_{\delta}(x^*,\rho^*)\right) =
Q^{\lambda, \widetilde{V} }_{x_0,\s_0}\left(\frac{d Q^{\lambda}_{x_0,\s_0}}{d
Q^{\lambda,\widetilde{V}}_{x_0,\s_0}}\chi_{B_{\delta}(x^*,\rho^*)}\right)
 \geq\\  Q^{\lambda,\widetilde{V}}_{x_0,\s_0}\left(B_{\delta}(x^*,\rho^*)\right)\inf_{(x,\rho)\in
B_{\delta}(x^*,\rho^*)}\frac{d Q^{\lambda}_{x_0,\s_0}}{d
Q^{\lambda,\widetilde{V}}_{x_0,\s_0}}(x,\rho)\,.
\end{multline}
Due to the particular choice of $\widetilde{V}$ (see Section
\ref{var_problema}), the law of large numbers stated in Proposition
\ref{LLN} implies that
$$
\lim_{\lambda \to
\infty}Q^{\lambda,\widetilde{V}}_{x_0,\s_0}
\left(B_{\delta}(x^*,\rho^*)\right)=1\,.
$$
Hence,  we can derive from  (\ref{alice}) and (\ref{speriamoinbene})
that
$$
\liminf_{\lambda \to \infty}\frac{1}{\l} \log
Q^{\lambda}_{x_0,\s_0}\left(O\right)\geq -\sup_{(x,\rho)\in
B_{\delta}(x^*,\rho^*)}J_{\widetilde V}(x,\rho)\,.
$$
Due to the fact that this bound holds for any $\delta$ small enough
we also have that
$$
\liminf_{\lambda \to \infty}\frac{1}{\l} \log Q^{\lambda}_{x_0,\s_0}
\left(O\right)\geq -\lim_{\delta \to 0}\sup_{(x,\rho)\in
B_{\delta}(x^*,\rho^*)}J_{\widetilde V}(x,\rho)\,.
$$
From the continuity in $(x,\rho)$ of $J_{\widetilde V}(x,\rho)$ we
deduce that
$$
\lim_{\delta \to 0}\sup_{(x,\rho)\in
B_{\delta}(x^*,\rho^*)}J_{\widetilde V}(x,\rho) =J_{\widetilde V}(x^*,\rho^*)=J(x^*,\rho^*)
$$
(the last identity follows from the definition of $\widetilde V$  and the
results of Section \ref{var_problema}).

 Optimizing over all possible $(x^*,\rho^*)\in O \cap \cB $
we finally get
$$
\liminf_{\lambda \to \infty}\frac{1}{\l} \log
Q^{\lambda}_{x_0,\s_0}(O)\geq -\inf_{(x,\rho)\in O\cap \cB
}J(x,\rho)\,.
$$
In order to conclude the proof of the lower bound we only need to
show that
\begin{equation}\label{arrivederci}
\inf_{(x,\rho)\in O\cap \cB }J(x,\rho)=\inf_{(x,\rho)\in
O}J(x,\rho)\,.
\end{equation} Trivially, the l.h.s. is not
smaller than the r.h.s.  In order to prove the opposite inequality,
fix $(x,\rho)\in \Upsilon$ and fix a sequence $( x^{(n)}, \rho
^{(n)})\in \cB$ converging to $(x,\rho)$ in $\Upsilon$. By the
construction of this approximating sequence given
 in the proof of Lemma \ref{ucraina}, we can assume that
 there exists a set $U \subset [0,T]$ whose
complement has zero Lebesgue measure such that $\rho ^{(n) }_\s (s)
\rightarrow \rho _\s (s)$ and $ x^{(n) }(s) \rightarrow x(s)$ for
all $s \in U$ and $\s \in \G$. Due to the continuity of the
transition rates, this implies  that
\begin{equation}\label{convergo} c^{(n) } (s)[\s,\s']:=
\rho^{(n)}_\s (s) r(\s,\s'|x^{(n)}(s),s) \rightarrow \rho_\s (s)
r(\s,\s'|x(s),s)=:c(s)[\s,\s']\,, \end{equation} for all $s \in U$
and all  $\s,\s'\in \G$. Recall the function $\cJ$ defined in
\eqref{emma28},  Section \ref{var_problema}. Then,
 due
to (\ref{convergo}) and the continuity of $\mathcal{J}$ (see Lemma
\ref{pierbau}), we obtain that
\begin{equation}\label{dentino}
j\bigl( \rho^{(n)} (s), r(\cdot,\cdot|x^{(n)}(s),s) \bigr)=
\mathcal{J} \bigl( c^{(n)}(s) \bigr) \rightarrow
\mathcal{J}\bigl(c(s) \bigr) =j\bigl( \rho(s), r(\cdot,\cdot|x(s),s)
\bigr)
\end{equation}
for each $s\in U$. Since $\|x^{(n)} - x \|_\infty $ goes to zero as
$n\uparrow \infty$ and due to the continuity of the transition
rates, given $\e>0$ we can find $n _0$ such that
$$ \sup _{n\geq n_0}\sup_{s\in [0,T] }\sup _{\s,\s'} r(\s,\s'|x^{(n)}(s),s) \leq
\sup _{s \in [0,T]}\sup _{\s,\s'} r(\s,\s'|x(s),s) +\e=:C\,.$$ Due
to the definition \eqref{pranzo} of $j$, this implies that
\begin{equation}\label{canino}
\begin{cases}
j\bigl( \rho^{(n)} (s), r(\cdot,\cdot|x^{(n)}(s),s) \bigr)\leq C |\G|^2 \,,\\
j\bigl( \rho (s), r(\cdot,\cdot|x (s),s) \bigr)\leq C|\G|^2\,,
\end{cases}
\end{equation}
for all $ s$ in  $[0,T]$ and $ n \geq n_0$. Now, due to
(\ref{dentino}), (\ref{canino}) and the dominated convergence
theorem we can conclude that
$$ J\bigl( x^{(n)}, \rho^{(n)}\bigr) = \int_0^T j\bigl( \rho^{(n)} (s), r(\cdot,\cdot|x^{(n)}(s),s)
\bigr)ds \rightarrow \int_0^T j\bigl( \rho (s), r(\cdot,\cdot|x
(s),s) \bigr)ds = J(x,\rho)\,,$$ thus implying (\ref{arrivederci}).

\section{Proof of Theorem \ref{LLNws}}\label{paccolone}

We first prove that  the family $R^\l _{x_0, \s_0}$ is relatively
compact  and then characterize its limit points.

\smallskip

\noindent $\bullet$ {\sl  Relative compactness}. We use Prohorov
theorem and Aldous compactness criterion (see for example
\cite{KL}[Section 4.1]). Since $\G^{(\ell)}$ is compact (endowed with
the discrete topology), we only have to check that
\begin{enumerate}

\item

For each $t \in [0,T]$ and $\e>0$, there exists a compact $\cK
\subset \bbR^d$ such that
\begin{equation}\label{liberi}
 R^\l
_{x_0, \s_0} ( x(t) \in \cK) = P^\l _{x_0,\s_0} (x(t) \in \cK ) \leq
\e \,, \qquad \forall \l> 0\,,
\end{equation}

\item

and that
\begin{equation}\label{tutti}
\begin{split}
\lim _{\theta \downarrow 0}\, \limsup_{\l \uparrow \infty}\, \sup
_\t
  R^\l _{x_0,\s_0} \Bigl( |x(\t)-x\bigl((\t+\theta)\wedge & T
\bigr)|>\epsilon\\
&\text{ or } \a(\t)\not = \a \bigl( (\t+\theta)\wedge T\bigr) \Bigr)
=0\,,
\end{split}
\end{equation}
where $\tau$ varies among all stopping times bounded by $T$.
\end{enumerate}

   As observed in  Lemma \ref{phelps}, there exists a
compact $\cK \subset \bbR^d$ such that $x(t) \in \cK$ for all $t \in
[0,T]$ $P^\l_{x_0,\s_0}$--a.s. and   for all $\l>0$. Since $|x(t)-x(s)|=\bigl| \int_s ^t
F_{\s(u) } (x(u),u)du\bigr|$, we conclude that there exists a
positive constant $c>0$ such that $|x(t)-x(s)|\leq c (t-s)$ for all
$s<t$ in $[0,T]$, $P^\l_{x_0, \s_0}$--a.s and for all $\l>0$. Hence,
we only need to prove that
\begin{equation}\label{risata}
\lim _{\theta \downarrow 0}\, \limsup_{\l \uparrow \infty}\, \sup
_\t
   R^\l _{x_0,\s_0} \Bigl(  \a(\t)\not = \a \bigl( (\t+\theta)\wedge T\bigr) \Bigr)
=0\,.
\end{equation}
Below we restrict to $\theta \in [0,1]$, moreover we use the shorter
notation  $\t_\theta := (\t+\theta)\wedge T$.
 We can write
\begin{equation}\label{foto1}
R^\l _{x_0,\s_0} \left( \a(\t)\not = \a \bigl( (\t+\theta)\wedge
T\bigr) \right)=
E^\l_{x_0,\s_0} \left[
  P^\l _{x_0, \s_0} \Bigl( \alpha  (\t)  \not = \alpha(\t_\theta)  \,|\, x(\t) ,\s(\t), \t\Bigr) \right]
 \,.
\end{equation}
Due to the strong Markov property of   PDMPs,   we can bound
\begin{multline}\label{foto2}
P^\l _{x_0, \s_0} \left( \a(\t) \not = \a (\t_\theta )
\, |\, x(\t)=x ,\s(\t)=\s, \t=t \right)\leq\\
 P^\l_{x_0,\s_0} \left(
\exists s \in (\t, \t_\theta ] \text{ s.t. } \a (s) \not =
 \a(\t) \,|\,x(\t)=x, \s(\t) = \s, \t =t\right)\leq \\
 P^\l_{x,\s, t} \left(\exists s
\in (t,(t+ \theta)\wedge T] \text{ s.t. } \s(s)\not \in
\G_{\a(\s)}\right)\,,
\end{multline}
where $P^\l_{x,\s, t}$ denotes the law of the $\l$--rescaled PDMP
starting in $(x,\s)$ at time $t$.

Due to the discussion at the beginning, we know that $x(\t)$ belongs
to  $ \cK$, thus implying that in the above expression we can
restrict to  points $x $ belonging to $\cK$. Similarly, starting in
$x \in \cK$ the $\l$--rescaled process  with law $P^\l_{x,\s,t}$
cannot leave a fixed compact $\cK'$ (independent of $\l,x, \s,t$)
in time $\theta \leq 1$. Hence, defining
$$c:= \sup _{x'\in \cK'}\sup_{u \in [0,T]} \max _{\s_1 \in
\G}\sum _{\s_2\in \G: \s_2\not=\s_1} r(\s_1,\s_2|x', u) \,,
$$
  the number $N$ of chemical jumps for the process $P^\l_{x,\s, t}$ in the interval $(t,(t+\theta)\wedge T ]$ is
stochastically dominated by a Poisson variable $\widehat N$ of mean
$c\l\, \theta $, uniformly in  $(x,\s,t)\in \cK\times \G\times [0,T]
$ and $\l\geq 1$. By similar arguments, whenever  the process with
law $P^\l _{x,\s,t}$ makes a chemical jump  in the time interval
$(t,(t+\theta)\wedge T ]$,   the probability that the jump is
between different chemical
 metastates is bounded from above by $C /\l$, $C$ not depending on
$(x,\s,t)\in \cK \times \G\times [0,T]$, $\l\geq 1$.
 Therefore, conditioned to make $n$ chemical  jumps in the time
interval $(t, (t+\theta)\wedge T]$, the probability that at least
one jump is between different chemical metastates is bounded from
above by $Cn/\l$. Hence, we can estimate
\begin{multline}
P^\l _{x,\s,t}\left(\exists s \in [t,(t+ \theta)\wedge T] \text{
s.t. } \s (s) \not \in \G_{\a(\s)}\right)\leq\\ \sum _{n=1}^\infty P
^\l_{x,\s,t}( N = n ) C n /\l \leq \sum _{n=1}^\infty P (\widehat N
= n ) C n /\l = (C/\l) E(\widehat N) = c \, C\, \theta\,.
\end{multline}
This allows to bound the r.h.s. of (\ref{foto2}) by $cC\theta$
uniformly in $\lambda\geq 1$, thus implying the same bound for the
first expression in (\ref{foto1}). This concludes the proof of
(\ref{risata}) and therefore the proof of the relative compactness
of $\{R^\l_{x_0,\s_0}\}_{\l>0}$.

\bigskip

\noindent $\bullet$ {\sl Characterization of the limit points}.
Given a path $\s(t)$, define the times $T_1, T_2, \dots $ as the
consecutive times in $[0,T]$ at which the system jumps between
different metastates, i.e.
\begin{align*}
 T_1 & = \inf  \left\{ t \in [0,T] \,:\, \a (\s(t) ) \not = \a (\s(0)
)\right \}\,, \\
T_k & = \inf \left\{t \in (T_{k-1},T] \,:\, \a( \s(t)) \not= \a (\s
(T_{k-1}) )\right\}\,, \qquad k \geq 2\,,
\end{align*}
with the convention that $T_k=\infty$ if $k$ is larger than the
number of jumps in the time interval $[0,T]$ between different
metastates.
 Fix  $(x_0, \s_0)\in \bbR^d
\times \G$, a sequence $0<t_1<t_2<\cdots < t_n<T$   and fix $\s_1,
\s_2, \dots , \s_n$ such that $ \a( \s_i ) \not = \a( \s_{i+1})$ for
each $i=0,1, \dots ,n -1$. In addition, fix $\d>0$ and $\e>0$ small
enough that $\e<T-t_n$ and $ \e<t_{i+1}-t_i$ for each $i=0,1,\dots,
n-1$, where $t_0:=0$. Then define the event
\begin{equation}\label{seriosotto} \cA=\left\{T_k \in (t_k-\e,
t_k+\e) \text{ and } \s( T_k)= \s_k \,, \forall k= 1,2, \dots, n
\right\}\cap \{ \sup_{t \leq T_n} |x(t) -x_*(t) | <\d\} \,,
\end{equation}
where the path $\bigl(x_*(t):  t \in [0,T_n]\bigr)$ is the only
continuous path in $\bbR^d$  such that
\begin{equation}\label{ribelle}
\dot{x}_* (t) = F_{ \a ( \s_k ) }  ( x_*(t), t ) \qquad \forall t
\in [T_k, T_{k+1}] ,\;\; k=0,1,\dots, n-1 \,. \end{equation} In the
above formula $T_0:=0$ and the vector field $F_i$ for $i \in
\G^{(\ell)}$ is the one defined in (\ref{babbol}).
 We claim that
\begin{multline}\label{serioso}
 \lim  _{\l \uparrow \infty} P^\l _{x_0, \s_0}(\cA)= \int_U du_1 du_2 \dots d u_n
 \prod _{k=0}^{n-1}\Bigl(e^{- \int _{u_k}
^{u_{k+1}} \g_{\a (\s_k) }(x_*(s),s)ds }\,\times\\
 \sum_{\widehat \s _k\in
\G_{\a(\s _k)}}    \mu _{\a(\s_k) } (\widehat \s _k |
x_*(u_{k+1}),u_{k+1}  ) r(\widehat \s_k ,\s_{k+1}
|x_*(u_{k+1}),u_{k+1}) \Bigr)\,,
\end{multline}
where $u_0:=0$, $ U := \prod _{i=1} ^n (t_i-\e, t_i +\e)$ and
\begin{equation}\label{solstizio} \g_i(x,s) := \sum _{j  \in
\G^{(\ell)} :j\not = i} r(i,j|x,s), \qquad i \in \G^{(\ell)}\,.
\end{equation}
We first prove the above claim for $n=1$.  For simplicity of
notation we write $t$ in place of $t_1$ and $\s'$ in places of
$\s_1$, while we call  $\s_0 , \s_1, \s_2, \dots , \s_r$  the finite
sequence of states in $\G_{\a(\s_0)}$ visited by the process before
jumping to $\s'$. Then, defining now
\begin{equation}\label{seriosobis}
\cA= \left\{T_1 \in (t-\e,t+\e) \,, \s (T_1)= \s'\,, \sup _{ s\leq
T_1 } |x(s)-x_*(s)|<\d\right\}\,,
\end{equation}
   we can write
\begin{multline}\label{socrate}
 P^\l_{x_0,\s_0} ( \cA)=
 \sum_{r=0}^\infty \sum _{\s_1, \s_2, \dots, \s_r}
\int_{t-\e}^{t+\e} du \int _0 ^u d \t_1 \int_{\t_1} ^u d \t _2
\cdots \int_{\t_{r-1} }^u
d\t_r\\
  \exp \Bigl\{ - \sum _{k=0}^{r}
\int_{\t_k} ^{\t_{k+1}}\bigl[\l \widehat \gamma _{\s_k} ( \widetilde
x(s) ,s) + b_{\s_k} ( \widetilde{x} (s),s) \bigr] ds
\Bigr\}\\\left[\prod _{k=0}^{r-1} \l r(\s_k,
\s_{k+1}|\widetilde{x}(\t_{k+1}), \t_{k+1})\right] r(\s _r , \s'
|\widetilde{x} (u), u)\chi\left( \sup _{ s\leq u }  |\widetilde
x(s)-x_*(s)|<\d \right)\,,
\end{multline}
where in the above expression $\t_0:=0$, $\t_{r+1}:=u$,
$\{\widetilde x(s): s \in [0,u]\} $ is the only continuous path on
$[0,u]$ starting in $x_0$  such that
$$
\dot{\widetilde x} (s) = F_{\s_k} (\widetilde x(s), s) \,,\qquad
\forall s \in ( \t_k, \t_{k+1} ) \,, \; \forall k = 0,1, \dots, r
$$
 and, for $\s \in \G_{\a (\s_0) } $,
\begin{align*}
&  \widehat \g _\sigma (x,s) = \sum_{\widehat \s \in \G _{\a (\s_0)}
} r(\s,
\widehat \s | x,s) \,,\\
&  b  _\sigma (x,s) = \sum_{\widehat \s \in \G \setminus \G _{\a
(\s_0)} } r(\s, \widehat \s | x,s)\,.
\end{align*}
Let us call $\widehat P ^\l _{x_0, \s_0}$ the $\l$--rescaled PDMP
with chemical states in $\G_{\a(\s_0)}$, transition rates $\l r (\s
, \s' |x,s) $ and vector fields $F_\s (x,s)$, $\s,\s' \in \G_{
\a(\s_0)}$. We write $\widehat E^\l _{x_0, \s_0}$ for the associated
expectation. Then it is simple to check that the r.h.s. of
(\ref{socrate}) equals
\begin{equation}\label{acquachiara}
\widehat E^\l _{x_0, \s_0}
 \left[
\int _{t-\e} ^{t+\e} du\, \exp \left\{ -\int _0 ^u  b _{\s(s)}
(x(s),s)ds \right\} r(\s(u), \s'|x(u),u)\chi\bigl( \sup _{ s\leq u }
|x(s)-x_*(s)|<\d \bigr)\right]\,.
\end{equation}
(Above we have used that $\s(u)= \s(u-)$ $\widehat P^\l_{x_0,
\s_0}$--a.s.) One can compute the limit of (\ref{acquachiara}) as
$\l \uparrow \infty$ by means of the LLN given in Theorem \ref{LLN},
applied to $\widehat P^\l_{x_0, \s_0}$.
 Indeed, we know that for each
$\b>0$ $\widehat P ^\l _{x_0, \s_0}\bigl( \sup _{ s\leq t+\e }
|x(s)-x_*(s)|<\b \bigr) \rightarrow 1$. This allows to write
$$ P^\l_{x_0, \s_0}(\cA)=
\widehat E^\l _{x_0, \s_0}
 \left[
\int _{t-\e} ^{t+\e} du\, \exp \left\{ -\int _0 ^u b _{\s(s)}
(x_*(s),s)ds \right\} r(\s(u), \s'|x_*(u),u)\right]+ o(1)\,,
$$
where here and below we denote  $o(1)$ any quantity such that $$
\lim _{\b\downarrow 0 } \limsup _{\l \uparrow \infty} o(1) =0\,.$$
It is simple to derive from Theorem \ref{LLN} that
$$
 \lim _{\l
\uparrow \infty}\widehat P^\l _{x_0, \s_0} \left(\sup _{t-\e\leq u
\leq t+\e} \left|\int _0 ^u b _{\s(s)} (x_*(s),s)ds- \int_0 ^u \g_{\a(
\s_0)} (x_*(s),s) ds
 \right|>\b
\right)=0\,,
$$
since (recall (\ref{solstizio}))
\begin{eqnarray*}
 \gamma_{\a(\s_0) } (x_*(s),s)& = &
\sum _{\widehat \s \in \G_{\a(\s_0) } }
 \sum_{\widetilde \s \in    \G \setminus \G_{\a(\s_0)} }
\mu_{\a(\s_0) } (\widehat \s |x_*(s),s) r(\widehat \s , \widetilde
\s |
x_*(s),s) \\
&=& \sum _{\widehat \s \in \G_{\a(\s_0) } } \mu_{\a(\s_0) }
(\widehat \s |x_*(s),s) b_{\widehat \s}(x_*(s),s)\,.
\end{eqnarray*}
Hence we can write
$$ P^\l_{x_0 , \s_0} (\cA)=
\widehat E^\l _{x_0, \s_0}
 \left[
\int _{t-\e} ^{t+\e} du\, \exp \left\{ -\int _0 ^u \gamma_{\a(\s_0)}
(x_*(s),s) ds  \right\} r(\s(u), \s'|x_*(u),u)\right]+ o(1)\,,
$$
At this point, one can apply again Theorem \ref{LLN}  and conclude
that \begin{multline*}
 P^\l _{x_0, \s_0 }(\cA)=\\
\sum _{\widehat \s \in \G_{\a(\s_0)} } \int _{t-\e} ^{t+\e} du\,
\exp \left\{ -\int _0 ^u \gamma_{\a(\s_0)} (x_*(s), s) ds \right\} \mu
_{\a(\s_0) } (\widehat \s | x_*(u),u ) r(\widehat \s ,
\s'|x_*(u),u)+o(1)\,.
\end{multline*}
By taking the limit $\l \uparrow \infty$ and afterwards using the
arbitrariness of $\beta$, we then obtain
\begin{equation}\label{banana}
\begin{split}
 \lim _{\l \uparrow \infty} & P^\l _{x_0, \s_0 }(\cA)=\\
& \sum_{\widehat \s \in \G_{\a(\s_0)} } \int _{t-\e} ^{t+\e} du\,
\exp
 \left\{ -\int _0 ^u \gamma_{\a(\s_0) } (x_*(s),s) \right\} \mu
 _{\a(\s_0) } (\widehat \s | x_*(u),u ) r(\widehat \s , \s'|x_*(u),u)\,.
\end{split}
\end{equation}
This concludes the proof of (\ref{serioso}) when $n=1$. Let us now
show how to prove (\ref{serioso}) when the event $\cA$ is defined as
in (\ref{seriosotto}) with $n=2$. The general case is completely similar.
By the strong Markov property, we
can write \begin{equation}\label{nido}
\begin{split}
 P^\l _{x_0, \s_0 }(\cA)=   E^\l_{x_0, \s_0}
\Bigl[ & \chi \Bigl\{ T_1 \in (t_1-\e, t_1+\e)\,, \,   \s(T_1)=\s_1
,\, \\
&\sup_{t\leq T_1}|x(t)-x_*(t)|\leq \d \Bigr\} f^\l\bigl(x(T_1), \s_1,T_1\bigr)\Bigr]\,,
\end{split}
\end{equation}
where, for $s\leq t_2-\e$,
$$ f^\l(x',\s',s):=  P^\l_{x',
\s',s } \Bigl(T_1 \in (t_2-\e, t_2+\e)\,, \; \s (T_1)=\s_2
\,,\;\sup_{s\leq t\leq T_1}|x(t)-z_*(t)|\leq \d\Bigr) \,,
$$
 $z_*(t)$ being the path starting in $x'$ at time $s$,
 such that $\dot{z}_* (t)=F_{\a (\s')} \bigl(z_* (t), t\bigr)$.
Due to  (\ref{banana}) with modified starting state and starting time, we know that
$f^\l(x',\s',s)$ converges to $g (x',\s', s)$ defined as
\begin{multline}
g(x',\s',s):=\\
 \sum_{\widehat \s \in \G_{    \a(\s') } }
 \int _{t_2-\e} ^{t_2+\e} du\, \exp
 \left\{ -\int _s ^u \gamma_{\a(\s') } (x_*(s),s) \right\} \mu
_{\a(\s') } (\widehat \s | x_*(u),u ) r(\widehat \s ,
\s_2|x_*(u),u)\,.
\end{multline}
 By simple arguments (as the ones used in the proof of Lemma \ref{funghetto}) one can improve \eqref{banana} and
 conclude that, given a compact $\cK$,
 $$ \lim _{\l\uparrow \infty} \sup _{x' \in \cK, \s'\in \G, s \in [0,t_2-\e] } \Bigl|
 f^\l(x',\s', s)- g ( x', \s', s) \Bigr|=0\,.
 $$
This allows to replace in  (\ref{nido}),  $f^\l\bigl(x(T_1), \s_1,T_1\bigr)$ with
 $g\bigl(x(T_1), \s_1,T_1\bigr)$ plus a negligible error as $\l \uparrow \infty$.
  The conclusion of the proof of  (\ref{serioso}) for $n=2$
   follows now from the LLN of Theorem \ref{LLN} by the same arguments used in the proof of (\ref{banana}).

\bigskip
Having proved (\ref{serioso}), we derive from it the following fact.
Fix  $(x_0, \s_0)\in \bbR^d \times \G$, a sequence $0<t_1<t_2<\cdots
< t_n<T$   and fix $i_1, i_2, \dots , i_n\in \G^{(\ell)}$ such that
$ i_k \not = i_{k+1}$ for each $k=0,1, \dots ,n-1 $ ($i_0:= \a
(\s_0) $). In addition, fix $\d>0$ and $\e>0$ small enough that $\e
< T-t_n$ and $ \e<t_{i+1}-t_i$ for each $i=0,1,\dots, n-1$ where
$t_0:=0$. Then define the event $\cC=\cC(\e,\d)$ as
\begin{equation} \cC=\left\{T_k \in (t_k-\e,
t_k+\e) \text{ and } \a( \s(T_k) )= i_k \,, \forall k= 1,2, \dots, n
\right\}\cap \{ \sup_{t \leq T_n} |x(t) -x_*(t) | <\d\} \,,
\end{equation}
where the path $\bigl(x_*(t): t \in [0,T_n ]\bigr)$ is the only
continuous path in $\bbR^d$ such that \begin{equation}\dot{x}_* (t)
= F_{ i_k} ( x_*(t), t ) \qquad \forall t \in [T_k, T_{k+1}] ,\;\;
k=0,1,\dots, n-1 \,.\end{equation}
 Then, summing in (\ref{serioso}) over $\s_1, \s_2, \dots , \s_n$,
we obtain
\begin{multline}\label{rosita}
\lim _{\l \uparrow \infty} P^\l _{x_0, \s_0} (\cC)=
  \int_U du_1 du_2 \dots d u_n \prod _{k=0}^{n-1}\left(e^{- \int _{u_k}
^{u_{k+1}} \g_{i_k}(x_*(s),s)ds } r(i_k , i_{k+1} |x_*( u_{k+1}),
u_{k+1} ) \right)\,.
\end{multline}
Since $\cC$ is an open subset of the Skohorod space $D\bigl( [0,T],
\bbR^d \times \G^{(\ell)} \bigr)$, we derive that for any limit
point $\cR$ of the family $\{\cR^\l _{x_0, \a( \s_0 )}:\l>0\} $ the
probability $\cR(\cC)$ is not larger than the r.h.s. of
\eqref{rosita}.  By similar arguments, one obtains that the limit
$P^\l_{x_0,\s_0}(\cC')$ coincides with the r.h.s. of \eqref{rosita},
where the event $\cC'=\cC'(\e,\d)$ is defined as
\begin{equation} \cC'=\left\{T_k \in [t_k-\e,
t_k+\e] \text{ and } \a( \s(T_k) )= i_k \,, \forall k= 1,2, \dots, n
\right\}\cap \{ \sup_{t \leq T_n} |x(t) -x_*(t) | \leq \d\} \,.
\end{equation}
Since $\cC'$ is closed, we derive that $\cR(\cC')$ is not smaller
than the r.h.s. of \eqref{rosita}. On the other hand, $\cC'(\e',\d')
\subset \cC(\e,\d)$ for $\e' <\e$ and $\d'<\d$, hence
$\cR(\cC'(\e',\d')  ) \leq \cR(\cC(\e,\d) )$. By taking the limits
$\e' \rightarrow \e$ and  $\d'\rightarrow \d$,  the above
observations implies that
\begin{equation}
\cR(\cC)=
  \int_U du_1 du_2 \dots d u_n \prod _{k=0}^{n-1}\left(e^{- \int _{u_k}
^{u_{k+1}} \g_{i_k}(x_*(s),s)ds } r(i_k , i_{k+1} |x_*( u_{k+1}),
u_{k+1} ) \right)\,.
\end{equation}

The above family of  identities parameterized by $t_1, \dots, t_n$,
$\e$ and $\d$ allows to conclude that there exists a unique limit
point and it must coincide with the law of the PDMP described in
Theorem \ref{LLNws}.

\appendix

\section{Some topological  properties of the space $\Upsilon$}

For the reader's convenience, in this Appendix we collect some
properties of the metric space $\Upsilon$  that are used in the
text.  We stress that the definition of $\Upsilon$ given in
(\ref{Uil}) depends on the fixed initial mechanical state $x_0$.  We
call $\cM_*[0,T]\subset \cM[0,T]$ the image of the map $ L[0,T]\ni f
\rightarrow f(t) dt \in \cM  [0,T]$ and we define
$\cM_*[0,T]^{\G,1}$ as the set of positive measures
$\left(\rho_\s(t)dt\right)_{\s\in \G}$ such that $\sum_{\s\in
\G}\rho_\s(t)=1 \ a.e.$

\begin{Le}\label{announodue}
 Given $\rho \in \cM_*[0,T]^{\G,1}$, there exists a unique
$x(t)\in C[0,T]$ such that
\begin{equation}
x(t)=x_0+\sum_{\s\in \G}\int_0^tF_\s(x(s),s)\rho_\s(s)ds \qquad \forall
t\in[0,T]\,.\label{micien3}
\end{equation}
Moreover,
the map that associates to each  $\rho \in \mathcal
\cM_*[0,T]^{\G,1}$
 the unique element $x(t)\in C([0,T])$ satisfying
\eqref{micien3} is continuous.
\end{Le}
\begin{proof}

Due to \eqref{muccamuu}, if $x(t)$ solves \eqref{micien3} it must be
$|x(t)| \leq |x_0|+ c_1 t+ c_2 \int _0^t |x(s)|ds$. Then, applying Gronwall inequality, the path $x(t)$ must  lie
 inside a compact $\cK$, depending only on $x_0$. We fix the constant $K$ as in
\eqref{lipF}.

Let us define $\cQ$ as the subset
$$\cQ:=\{ \rho  \in\cM_*[0,T]^{\G,1}\,:\, \rho_\s \in C[0,T]\; \forall \s \in \G \}\,.$$
It is simple to check that $\cQ$ is dense in $\cM_*[0,T]^{\G,1}$. Indeed,
by using mollifiers, one can show that for each $\s\in \G$ there exists a
sequence $\rho^{(n)}_\s(s) \in C[0,T]$ such that $\rho^{(n)}_\s
(s)$ converges to $\rho_\s(s)$ in $L^1[0,T]$ as $n\uparrow \infty$.
At  cost to
  normalize, we can assume that $\sum _\s
\rho^{(n)} _\s (s)=1$ for each $s \in [0,T]$.

 If $\rho\in \cQ$,
existence and uniqueness of \eqref{micien3} can be proven by the same arguments used in
the proof of Lemma \ref{sibillino}, given in Appendix \ref{misciotto}.  In order to prove existence for \eqref{micien3} when $\rho \in\cM_*[0,T]^{\G,1}$, we take a sequence $\rho^{(n)}\in \cQ$ such that
$\rho^{(n)}_\s$ converges to $\rho_\s$ in $L^1[0,T]$ as $n\uparrow
\infty$.
  Consider the solutions
$x^{(n)}(t) \in C[0,T]$ associated to $\rho^{(n)} $.
We can bound
$$  \sum_{\s \in \G }
\left|\int_0 ^t\bigl[ \rho^{(n)} _\s (s)- \rho_\s ^{(m)} (s) \bigr] F_\s
(x^{(n)}(s),s) ds \right|\leq K_1 \sum_{\s \in \G }
\int_0 ^T\bigl| \rho^{(n)} _\s (s)- \rho_\s ^{(m)} (s) \bigr| ds =:C_{n,m} \,,$$
where $K_1:= \max _\s \max _{x \in \cK, t \in [0,T]} |F_\s (x,t)|$.
Due to \eqref{lipF}  we can estimate
\begin{equation}
\bigl| x^{(n)}(t) - x^{(m)}(t)\bigr|\leq C_{n,m}
+ K  \int _0 ^t |x^{(n)}(s) -x^{(m)}(s)|ds \,.
\end{equation}
Due to Gronwall lemma, we conclude that $ \|x^{(n)}(t) - x^{(m)}(t)\|_\infty\leq C_{n,m}e^{KT}$.
Since $C_{n,m}$ is arbitrarily small for $n,m$ large, we conclude that the sequence $x^{(n)}$ is a Cauchy sequence in
$C[0,T]$, and therefore it converges to some path $x\in C[0,T]$. Taking the limit $n\rightarrow \infty$ for equation
\eqref{micien3} with $x,\rho$ replaced respectively by  $x^{(n)}$, $\rho^{(n)}$, due to the Dominated Converge Theorem one concludes that $x(t)$ solves \eqref{micien3}.
Uniqueness follows from Gronwall inequality, since given two solutions  $x_1(t)$ and $x_2(t)$ of \eqref{micien3} it must be
$|x_1(t)-x_2(t)|\leq K \int_0 ^t |x_1(s)-x_2(s)|ds $.

Finally let us prove the continuity of the map $\cM_*[0,T]^{\G,1}\ni \rho(t) \rightarrow x(t) \in C[0,T]$.
We introduce a metric $D$ on $\cM_*[0,T]$ defined as
\begin{equation}\label{georgia}
D(f_1(t)dt ,f_2(t) dt ) = \sup _{t\in [0,T]}\left| \int _0 ^t \bigl[
f_1(s) -f_2(s) \bigr]ds \right| \,, \qquad f_1,f_2 \in L[0,T]\,.
\end{equation}
It is simple to check that $D$ is a distance on $\cM_*[0,T]$. We
claim that the topology induced by $D$ coincides with the weak
topology of $\cM_*[0,T]$. To this aim, we only need to show that $
D(f_n(t) dt, f(t)dt ) \rightarrow 0$ if and only if $ \int _0 ^T f_n
(t) g(t) dt \rightarrow \int_0 ^T f(t) g(t) dt $ for each $g \in
C[0,T]$. Given $h \in L[0,T]$ let $\widehat h (t) = \int_0 ^t h(s)
ds $ for $0\leq t \leq T$. Since $\widehat h$ is a function of
bounded variation, it is simple to check that $D (f_n(t) dt , f(t)dt
) \rightarrow 0$ if and only if $ \widehat f_n (t) \rightarrow
\widehat f (t)$ for each $t \in [0,T]$. The proof that the weak
convergence $f_n(t)dt \rightarrow f(t)dt $ coincides with the
pointwise convergence of $\widehat f _n $ to $\widehat f$ follows
the same arguments leading to the equivalence between the weak
convergence of probability measures and the pointwise convergence of
the associated distribution functions.

Take  a sequence
$\rho^{(n)}\in \mathcal M_*[0,T]^{\G,1}$ converging to some $\rho\in
\mathcal M_*[0,T]^{\G,1}$.
 Consider also the corresponding
$x^{(n)}(t),\ x(t) \in C[0,T]$ obtained from \eqref{micien3}.
Setting
$$C_n(t):= \sum_{\s \in \G }
\left|\int_0 ^t\bigl[ \rho^{(n)} _\s (s)- \rho_\s (s) \bigr] F_\s
(x(s),s) ds \right|\,,$$  due to \eqref{lipF} in assumption (A4) we
have
\begin{multline}
\bigl| x^{(n)}(t) - x(t)\bigr|\leq C_n(t)+ \sum_{\s \in \G}\int_0^t
\rho^{(n)}_\s(s) \bigl| F_\s (x(s),s)
-F_\s( x^{(n)}(s),s ) \bigr|ds\leq\\
C_n(t)+ K  \int _0 ^t |x^{(n)}(s) -x(s)|ds \,.
\end{multline}
Since $\rho ^{(n) }_\s (s) ds $ weakly converges  to $\rho _\s (s)
ds $ on $[0,T]$ and therefore on $[0,t]$ and  since $F_\s( x(s),s) $
is continuous in $s$,  from the previous results on the metric $D$
we conclude that $C_n(t) \rightarrow 0$ as
$n\rightarrow \infty$ uniformly in  $t \in [0,T]$.
 Due to the above
observations we conclude that  for each $\d>0$ there exists  $n_0$
such that for each $t \in [0,T]$ and $n \geq n_0$ it holds
$$
\bigl| x^{(n)}(t) - x(t)\bigr|\leq \d + K \int _0 ^t |x^{(n)}(s)
-x(s)|ds\,.
$$
By Gronwall lemma it follows that $ \| x^{(n)} - x \|_\infty \leq\d
e ^{K T} $ for all $ n \geq n_0$. This implies that $x^{(n)}(t) $
converges to $x(t)$ in the uniform norm of $C[0,T]$.

\end{proof}

\begin{Le}\label{annozero}
The space $\Upsilon$ is a compact Polish metric space.
\end{Le}
\begin{proof}
First we prove  that $\cM_*[0,T]^{\G,1}$ is compact. It is a subset
of the compact space $\{\rho \in \cM[0,T]^\G\,:\, \rho_\s[0,T]\leq
T\ \forall \s\in \G\}$, so that we just need to prove that it is
closed. We first prove that $ \cM_* [0,T]$ is closed. To this aim,
let $\mu_n $ be a sequence in $ \cM_* [0,T]$ converging to $\mu \in
\cM [0,T]$. Since given  $g\in C[0,T]$ it holds
 $ |\mu_n (g)| \leq \int_0 ^T
 |g(t)| dt$, by taking the limit we conclude that
$ |\mu (g)|  \leq \| g\|_{L^1 [0,T]}$. By density we obtain  that
the map $L^1[0,T] \ni g \rightarrow \mu (g) \in \bbR$ is a
continuous linear functional with norm bounded by $1$. Since the
dual space of $L^1[0,T]$ is given by $L^\infty [0,T]$ endowed of the
essential uniform norm, we can conclude that there exists
$h \in L^\infty [0,T]$ with $\|h\|_\infty \leq 1$ such that $\mu (g)
= \int _0 ^T g(t) h(t) dt $ for all $g \in C[0,T]$. Since
$\mu(g)=\lim_{n\uparrow \infty } \mu_n (g)  \geq 0$ for all $ g\in
C[0,T]$ with $g \geq 0$, the function $h$ must be nonnegative
 a.e. This proves that $\mu \in \cM _* [0,T]$ and $ d\mu / dt =h$.
Hence $\cM _* [0,T]$ is a closed subspace of $\cM [0,T]$. Consider
now a sequence $ \rho^{(n)} \in \cM_*[0,T]^{\G,1}$ converging to
$\rho $. Then necessarily for every $\s\in \G$ it holds $\rho_\s\in
\cM _* [0,T]$. Moreover for each $g \in C[0,T]$ it holds
$$ \int_0^T g(s) ds = \sum _{\s \in G} \int _0 ^T \rho_\s ^{(n)} (s)g(s)  ds\rightarrow
\sum_{\s \in \G} \int _0^T \rho_\s (s)g(s)  ds = \int _0 ^T \left(
\sum _{\s \in \G} \rho _\s (s) \right) g(s) ds\,.
$$
This implies that
$$\int_0^T g(s) ds=\int _0 ^T \left(
\sum _{\s \in \G} \rho _\s (s) \right) g(s) ds\,, \qquad \forall g
\in C[0,T]\,.$$ Hence, $\sum _{\s \in \G} \rho _\s (s)=1$ a.s. and
this means that $\rho\in\cM_*[0,T]^{\G,1}$.


\smallskip

Compactness of $\Upsilon\subseteq \cM[0,T]\times C[0,T]$ follows
from the fact that it is the graph of a continuous function defined
on a compact domain $\cM_*[0,T]^{\G,1}\subseteq \cM[0,T]$.
Completeness and separability of $\Upsilon$ follow  from the fact
that $\Upsilon$ can be thought as a   closed subset of the space $
C[0,T]\times \cM[0,T]^\G$, which is complete and separable.
\end{proof}

\begin{Le}\label{ucraina}
The set $\cB$  defined as
\begin{equation*}
\cB = \bigl \{ (x,\rho)  \in \Upsilon \,:\, \rho _\s \in C^1[0,T]
\text{ and }  \rho_\s(t)>0  \;\forall \s \in \G,\; t \in
[0,T]\bigr\}
 \end{equation*}
 is a dense subset of $\Upsilon$.
 \end{Le}
 \begin{proof}
Fix $(x, \rho) \in \Upsilon$. Then, by using mollifiers, one can show that for each $\s\in \G$ there exists a
sequence $\rho^{(n)}_\s(s) \in C^1[0,T]$ such that $\rho^{(n)}_\s
(s)$ converges to $\rho_\s(s)$ in $L^1[0,T]$ as $n\uparrow \infty$.
At cost to take $ \max\{1/n, \rho^{(n)}_\s (s) \}$ we can assume
that $\rho^{(n)} _\s$ is positive; at cost to
  normalize, we can assume that $\sum _\s
\rho^{(n)} _\s (s)=1$ for each $s \in [0,T]$.
 Since $L^1$--convergence is stronger than
$L[0,T]$--convergence (where $L[0,T]$ is endowed with the metric
defined by the r.h.s. of \eqref{georgia}), we obtain that
$\rho^{(n)}_\s$ converges to $\rho_\s$ in $L[0,T]$ as $n\uparrow
\infty$. Let us call $x^{(n)}$ the solution of the Cauchy problem
\begin{equation*}
\begin{cases}
 \dot{x}^{(n)}(t)= \sum _{\s \in \G} \rho ^{(n)} _\s (t) F_\s ( x^{(n)}(t),t) \qquad t \in [0,T]\,,  \\
x^{(n)}(0)=x_0\,. \end{cases}
\end{equation*}
 Note that $x^{(n)}$ is well defined due to Lemma \ref{announodue}.
 Then $(x^{(n)},
\rho^{(n)})$ belongs to $\cB$ and from Lemma \ref{announodue} we have
that $\|x^{(n)}-x \|_\infty\rightarrow 0$.
\end{proof}

\section{A uniform large deviation estimate}
\label{molofacciosec} In this appendix we prove formula
\eqref{fabri}, keeping the same notation introduced in the proof of
Lemma \ref{funghetto}. Formula \eqref{fabri}  follows immediately
from the next lemma, valid under the assumptions stated in Section
\ref{mod}.
\begin{Le} \label{molofaccio} For any $\d>0$ it holds
\begin{multline}\label{fabriz}
\limsup  _{u\rightarrow \infty} \sup_{(x,t) \in \mathcal K'\times
[0,T],\; \s ,\s' \in \G}
\\\frac{1}{u} \ln Q_{\s',x,t } \left[ \left| \frac{1}{u } \int_0 ^u
\bigl[ \chi ( \s( s ) =\s) - \mu (\s|x,t) \bigr]ds
 \right|\geq  \delta \right] <0\,.
\end{multline}
\end{Le}
\begin{proof} We fix some notation.
We call $\rho^{(u)}$ the empirical measure in the time interval
$[0,u]$ defined as
$$
\rho^{(u)}_\s:=\frac 1u \int_0^u\chi ( \s( s ) =\s) ds \, \qquad
\forall \s\in \G\, .
$$
This is an element of $\mathcal M^1(\G)$, the set of probability
measures on $\G$. Moreover, we write $\cE:=K'\times[0,T]\times
\mathcal M^1(\G)$. Note that $\cE$ is a Polish space. For any fixed
$\s'\in\G$ and for any time $u$ we call $\widehat{Q}_{\s'}^u$ the
map that associates to any measurable subset $S\subseteq \cE$ the
nonnegative  number
$$
\widehat{Q}_{\s'}^u(S):=\sup_{(x,t)\in \mathcal K'\times
[0,T]}Q_{\s',x,t}\Big(\Big\{(x,t,\rho^{(u)})\in
S\Big\}\Big)=\sup_{(x,t)\in \mathcal
K'\times[0,T]}Q_{\s',x,t}\Big(\rho^{(u)}\in S_{(x,t)}\Big)\,,
$$
where $S_{(x,t)}:=\left\{\rho\in \mathcal M^1(\G)\ :\ (x,t,\rho)\in
S\right\}$.


In order to obtain upper bounds on $ \widehat{Q}_{\s'}^u(S)$, we
proceed as follows. Given an arbitrary function $V:\G\to \mathbb R$
we introduce, likewise in \eqref{alig}, the perturbed rate for a
jump from $\s$ to another chemical state $\s'$ as
\begin{equation}
r^V(\s,\s'|x,t):=r(\s,\s'|x,t)e^{V_{\s'}-V_\s}\,.
\end{equation}
The law of the Markov chain on $\G$ with the above perturbed rates
and  with initial condition $\s'$ is called $Q_{\s',x,t}^{V}$.
 Using a result analogous to
\eqref{speriamoinbene} in this simpler framework and using the
notation introduced in Section \ref{var_problema}, we obtain that
\begin{multline*}
 \limsup  _{u\rightarrow \infty} \frac{1}{u} \ln\left[
\widehat{Q}_{\s'}^u\left(S\right)\right]  =\limsup  _{u\rightarrow
\infty} \frac{1}{u}\sup_{(x,t)\in \mathcal K'\times [0,T]} \ln\Big[
Q_{\s',x,t}\Big(\rho^{(u)}\in S_{(x,t)}\Big)\Big] \\
=\limsup  _{u\rightarrow \infty} \frac{1}{u}\sup_{(x,t)\in \mathcal
K'\times [0,T]} \ln\Big[ \mathbb
E_{Q^{V}_{\s',x,t}}\Big(\frac{dQ_{\s',x,t}}{dQ^{V}_{\s',x,t}}
\chi\Big(\rho^{(u)}\in S_{(x,t)}\Big)\Big)\Big] \\
 \leq\limsup  _{u\rightarrow \infty} \frac{1}{u}\sup_{(x,t)\in
\mathcal K' \times [0,T]} \ln\Big[ \sup_{\rho^{(u)}\in
S_{(x,t)}}\frac{dQ_{\s',x,t}}{dQ^{V}_{\s',x,t}}\Big]
\\\text{ } \qquad \qquad \qquad=-\inf_{(x,t)\in \mathcal K'\times [0,T]}\inf_{\rho\in
S_{(x,t)}}\widehat{\mathcal
J}\Big(c[\rho,r(\cdot,\cdot|x,t)],e^V\Big)\\=-\inf_{(x,t,\rho)\in
S}\widehat{\mathcal J}\Big(c[\rho,r(\cdot,\cdot|x,t)],e^V\Big)
\end{multline*}
 We can now optimize over the arbitrary
functions $V$ obtaining
\begin{equation}
 \limsup  _{u\rightarrow \infty} \frac{1}{u} \ln
\widehat{Q}_{\s'}^u(S)\leq -\sup_{V}\inf_{(x,t,\rho)\in
S}\widehat{\mathcal
J}\left(c[\rho,r(\cdot,\cdot|x,t)],e^V\right)\label{carocaro}
\end{equation}
 We note that given a compact subset
$C\subseteq \cE$ a finite open cover  $O_1, \dots, O_n$ of $C$, it
trivially holds that $ \widehat{Q}_{\s'}^u\left(C\right)\leq
\sum_{i=1}^n\widehat{Q}_{\s'}^u\left(O_i\right)$.  Hence we can
estimate
\begin{equation}\label{ade}
\limsup  _{u\rightarrow \infty} \frac{1}{u} \ln
\widehat{Q}_{\s'}^u(C)\leq \inf _{O_1,\dots, O_n} \max _{1\leq j
\leq n} \inf _V \sup _{(x,t,\rho)\in O_j}-\widehat{\mathcal
J}\left(c[\rho,r(\cdot,\cdot|x,t)],e^V\right)\,,
\end{equation}
where the first infimum is carried over all finite open covers
$\{O_1,O_2, \dots , O_n\}$ of $C$. In order to bound the above
r.h.s. we can apply Lemma 3.2 in \cite{KL}[Appendix 2]. Indeed, the
assumption of this lemma are fulfilled since, as  composition of
continuous functions (see Lemma \ref{pierbau}),  the map $$\cE \ni
(x,t,\rho)\in \cE\rightarrow \widehat{\mathcal
J}\left(c[\rho,r(\cdot,\cdot|x,t)],e^V\right)\in \bbR$$ is
continuous for any $V$. As result we obtain that
\begin{multline}\label{carocaro}
\limsup  _{u\rightarrow \infty} \frac{1}{u} \ln
\widehat{Q}_{\s'}^u(C)\leq - \inf _{(x,t,\rho)\in C} \sup_V
\widehat{\mathcal J}\left(c[\rho,r(\cdot,\cdot|x,t)],e^V\right)=\\ -
\inf_{(x,t,\rho)\in
C}j\left(\rho,r(\cdot,\cdot|x,t)\right)=j\left(\rho^*,r(\cdot,\cdot|x^*,t^*)\right)\,,
\end{multline}
where $(x^*,t^*,\rho^*)$ is a minimum point of the continuous
function $(x,t,\rho)\to j\left(\rho,r(\cdot,\cdot|x,t)\right)$ on
the compact set $C$.

The statement of the lemma now follows from two simple facts. First:
since the map $(x,t)\in \mathcal K'\times [0,T]\to \mu(\cdot|x,t)\in
M^1(\G)$ is continuous we have that for any fixed $\s\in\G$ the set
$$
C:=\left\{(x,t,\rho)\in \mathcal K'\times[0,T]\times \mathcal
M^1(\G)\ :\ |\rho_\s-\mu(\s|x,t)|\geq \delta\right\}
$$
is compact. Second: from Remark \ref{vivairem} we have that the
continuous function $j\left(\rho,r(\cdot,\cdot|x,t)\right)$ is
strictly positive on $C$ and therefore also its minimum.

\end{proof}

\section{Miscellanea}\label{misciotto}   In this last appendix, we prove Lemma \ref{sibillino} and we collect some
technical results frequently used in the paper.

\smallskip

\noindent
{\bf Proof of Lemma \ref{sibillino}}. We consider here only
\eqref{saratoga3}, since the Cauchy problem \eqref{saratoga0} can be treated similarly.
 First, we observe that due to \eqref{lipmu}  the field $\bar F (x,t)$ is
  continuous  on $\bbR^d \times [0,T]$,  is  locally Lipschitz w.r.t. $x$ and satisfies
  \eqref{muccamuu} with $\bar F $ instead of $F_\s$. Then, due to Picard Theorem, the Cauchy problem \eqref{saratoga3} has locally a unique solution.
We only need to show that there exists a global solution on $[s,T]$.   Given $b>0$ we define
$M(b)= \max \{ |\bar F (x,t) |\,:\, |x-x_0|\leq b \,, \; t \in [0,T]\}$. Then by Peano Theorem, there exists a $C^1$ solution
$x(t)$ of  \eqref{saratoga3} defined for $t \in  [s, s+\a] $, where $\a:= \min\{ T-s, b/M(b)\}$. Due to \eqref{muccamuu}, we know that
$$ b/M(b) \geq  b /( \k_1+\k_2 |x_0|+ \k _2 b )\,.$$ We take $b=b(x_0)$ large enough that $b/M(b)\geq 1/(2\k_2)$. This implies that the solution
of \eqref{saratoga3} exists always on the interval $[s, (s+ 1/(2 \k_2))\wedge T]$, which does not depend on $x_0$. By patching a finite number of paths,
one obtains the global solution of \eqref{saratoga3}.

\qed

\medskip

Let us now show another consequence of Assumption (A4):

\begin{Le}\label{phelps}
Given $\s$ an element of $D([0,T], \G)$ and   given $s
\in [0,T]$, let $x(t|x_0,s)$
  be the unique continuous and piecewise $C^1$
solution on $[s,T]$  of the ODE  $ \dot{x}(t) = F_{\s(t) }(x(t),t) $
 starting at $x_0$ at time $s$. Then, for each compact subset
 $\cK\subset \bbR^d$ there exists another compact subsect $\cK'\subset
 \bbR^d$, independent from the path $\s(t)$, such that
$$\{ x(t|x_0, s): s \leq t \leq T,\; x_0 \in \cK\} \subset \cK'\,.$$
The same thesis holds if one replace $x(t|x_0,s)$ with the $C^1$
solution $x_*(t|x_0,s)$ of the ODE  $\dot{x_*} (t) = \bar F
(x_*(t),t) $,  starting at $x_0$ at time $s$.
\end{Le}
\begin{proof}
 We give the proof only for $x(t|x_0,s)$ since  the other case can be
treated similarly. Due to \eqref{muccamuu}, we can bound
$$ |x(t|x_0,s)|\leq |x_0|+ \int _s ^t |F_{\s(u)}(x(u),u) |du \leq |x_0|+ \k_1
(t-s)+ \k_2 \int _s ^t |x(u|x_0,s) |du\,.
$$
Due to  the Gronwall inequality,  the l.h.s. is therefore bounded by
$(|x_0|+\k_1 T) e^{\k_2 T}$, thus implying the thesis.
\end{proof}

\begin{Le}
Assumptions (A2) and  (A3) imply   that $\mu(\s|x,t)$ is a $C^1$ function in $x$ and $t$, for each $\s \in \G$.
\end{Le}
\begin{proof}
We
 fix $(x_0,t_0) \in \bbR^d \times [0,T] $. By assumption (A2) and Perron--Frobenius Theorem,
 the right kernel  of the matrix $L_c(x_0,t_0)$ has dimension one and  is generated by the row  $\mu(\cdot|x_0,t_0)$ (below, we will play with
 row and column vectors, the interpretation will be clear from the context and will be understood). Hence $L_c(x_0,t_0)$
has
 rank $n-1$, where $n:= |\G|$.
 In particular, there exists $\s_0 \in \G$ such that all the $\s$--columns $v_1, v_2, \dots , v_{n-1}$, with $\s \not =\s_0$, of the matrix
 $L_c(x_0,t_0)$   are independent.
  We  define $\p$ as the canonical projection
$\p:\bbR^{\G} \rightarrow \bbR^{\G\setminus \{\s_0\}}$. Then
we consider the map
$$ F: \bbR^d \times [0,\infty) \times  \bbR ^\G \ni(x,t, z) \rightarrow ( \p(z \cdot L_c (x,t) ) ,  \sum _\s z_\s \mu (\s |x_0,t_0)-a ) \in \bbR^\G\,,
$$
where $a = \sum_\s \mu(\s|x_0,t_0)^2$.
Note that $F$ is $C^1$ and that $F\bigl(x_0,t_0 ,z_0)=0$ where $z_0:= \mu(\cdot |x_0,t_0) \bigr)$. We claim that the map $F(x_0,t_0, \cdot)$ has
invertible tangent map  $T_z F(x_0,t_0, z_0)$ in $z_0$. Indeed, $T_z(x_0,t_0,z_0)$ maps $z\in \bbR^\G$ in $z\cdot A\in \bbR^\G$, where $A$ is a
 $\G\times \G$--matrix whose columns are given - a part the order - by $v_1$, $v_2$,... , $v_{n-1}$ and $z_0$.
 We already know that $v_1, \dots, v_{n-1}$ are independent. Moreover, since $z_0\cdot L_c(x_0,t_0) =0$, $v_1, \dots , v_{n-1}$ are all orthogonal to $z_0$.
 This implies that all the columns of the matrix $A$ are independent, and therefore $A$ is invertible, thus proving our claim.

We can now apply the Implicit Function Theorem and derive that there exists an open neighborhood $U$ of $(x_0,t_0)$ and a $C^1$ map
 $h:U \rightarrow \bbR^\G$ such that $F(x,t, h(x,t))=0$ for all $(x,t)\in U$. This implies that $(h(x,t) \cdot L_c (x,t))_\s=0$ for all $\s \not = \s_0$.
 Since $L_c (x,t)\cdot \underline 1 =0$,  we know that
 $$ \sum _{\s\in \G} (h(x,t) \cdot L_c (x,t) )_\s=h \cdot  L_c(x,t) \cdot \underline 1=0 $$
thus implying that $(h(x,t)\cdot L_c (x,t))_{\s_0}=0$. Moreover, since $h(x,t)\cdot z_0=a\not = 0$, we conclude that
 $h(x,t)$ is a nonzero right eigenvector of $L_c (x,t)$ with eigenvalue $0$. By Perron--Frobenius Theorem
we have that
$$ \mu(\s|x,t) = \frac{ h_\s(x,t)}{\sum _{\s'} h_{\s'} (x,t)} \,, \qquad \forall (x,t) \in U\,.
$$
The above identity and the fact that $h$ is $C^1$ imply that $\mu(\s|\cdot, \cdot)$ is $C^1$.

\end{proof}
\bigskip

\bigskip

\noindent
 {\bf Acknowledgements}. The authors thank Prof. G. Jona--Lasinio for introducing them to the subject
 of molecular motors and for useful discussions. Moreover, they thank  Prof. E. Vanden--Eijnded for
 useful discussions. One of the authors, D.G.,
  acknowledges the support of the G.N.F.M. Young Researcher  Project ``Statistical Mechanics of Multicomponent
Systems".

\end{document}